\documentclass[a4paper]{amsart}
\usepackage{amsmath,amsthm,amssymb,enumerate,color}

\date{December 21, 2009}

\newcommand{\spa}{\operatornamewithlimits{span}} 

\newcommand{\norm}[1]{\left\Vert#1\right\Vert}
\newcommand{\abs}[1]{\left\vert#1\right\vert}
\newcommand{\set}[1]{\left\{#1\right\}}
\newcommand{\Real}{\mathbb{R}}

\newtheorem{thm}{Theorem}[section]
\newtheorem{prop}[thm]{Proposition}

\newtheorem{lem}[thm]{Lemma}

\theoremstyle{definition}
\newtheorem{defn}[thm]{Definition}
\newtheorem{rem}[thm]{Remark}

\numberwithin{equation}{section}

\author[P. R. Stinga]{Pablo Ra\'ul Stinga}
\address{Departamento de Matem\'aticas \\
          Facultad de Ciencias \\
          Universidad Au\-t\'o\-no\-ma de Madrid \\
          28049 Madrid, Spain}
\email{pablo.stinga@uam.es}

\author[J. L. Torrea]{Jos\'e Luis Torrea}
\address{Departamento de Matem\'aticas \\
          Facultad de Ciencias \\
          Universidad Au\-t\'o\-no\-ma de Madrid \\
          28049 Madrid, Spain}
\email{joseluis.torrea@uam.es}

\thanks{Research supported by Ministerio de Ciencia e Innovaci\'{o}n de Espa\~{n}a under project MTM2008-06621-C02-01}



\begin{document}

\title{Regularity theory for the fractional harmonic oscillator}

\begin{abstract}
In this paper we develop the theory of Schauder estimates for the fractional harmonic oscillator $H^\sigma=(-\Delta+\abs{x}^2)^\sigma$, $0<\sigma<1$. More precisely, a new class of smooth functions $C^{k,\alpha}_H$ is defined, in which we study the action of $H^\sigma$. In fact these spaces are those adapted to the operator $H$, hence the suited ones for this type of regularity estimates. In order to prove our results, an analysis of the interaction of the Hermite-Riesz transforms  with the H\"older spaces $C^{k,\alpha}_H$ is needed, that we believe of independent interest.
\end{abstract}

\maketitle

\section{Introduction}

For a given partial differential operator $L$, the analysis of its regularity properties with respect to H\"older classes is one of the tools employed in the theory to prove important facts about partial differential equations. Indeed, being a bit imprecise, it is well known that if $f$ is a H\"older continuous function with exponent $\alpha$, then the equation $-\Delta u=f$ has a unique solution $u$, whose second order derivatives belong to $C^\alpha$, and $\norm{u}_{C^{2,\alpha}}$ is controlled by $\norm{f}_{C^\alpha}$. This result was first applied to obtain classical solutions of second order elliptic equations of the form $Lu=f$ (see for instance \cite[Chapter~6]{Gilbarg-Trudinger}). Recently, and motivated by the obstacle problem for the fractional Laplacian, L. Silvestre proved in \cite{Silvestre CPAM} (see also his thesis \cite{Silvestre Thesis}) the regularity properties for the operator $(-\Delta)^\sigma$, $0<\sigma<1$, when acting on H\"older spaces. For more applications see \cite{Caffarelli-Salsa-Silvestre} and \cite{Caffarelli-Silvestre Evans-Krylov}.

Let $H$ be the most basic Schr\"odinger operator in $\Real^n$, $n\geq1$, the harmonic oscillator:
$$H=-\Delta+\abs{x}^2.$$
The fractional powers $H^\sigma$, $0<\sigma<1$, were introduced in \cite{Stinga-Torrea}.

The aim of this paper is to prove regularity estimates in H\"older classes for the fractional harmonic oscillator $H^\sigma$. For this purpose, we define new H\"older spaces $C^{k,\alpha}_H$, different than the classical H\"older spaces $C^{k,\alpha}$, in which the smoothness properties of $H^\sigma$ are analyzed, see Definition \ref{calfa} and Theorem A.

The classes $C^{k,\alpha}_H$ are the natural spaces associated to $H$. This becomes evident, for instance, in the fact that the Hermite-Riesz transforms have the expected behavior: they preserve them, see Theorem \ref{Thm:R y Calfa}. Also the fractional integrals produce a kind of ``inverse fractional derivative" process when acting in $C^{k,\alpha}_H$, see Theorem B.

Our estimates, together with Harnack's inequality for $H^\sigma$ proved in \cite{Stinga-Torrea}, are the basic regularity estimates one expects to get when having the fractional powers of a second order operator. Moreover, we were able to find the right spaces for which Schauder estimates are appropriated. We expect that the correct regularity estimates for nonlinear problems related to the fractional harmonic oscillator can be obtained in these spaces. Applications will appear elsewhere.

Let us introduce the definition of $H^\sigma$. For a function $f\in\mathcal{S}$ and $0<\sigma<1$, the fractional harmonic oscillator $H^\sigma$ is given by the classical formula
\begin{equation}\label{H sig heat}
H^\sigma f(x)=\frac{1}{\Gamma(-\sigma)}\int_0^\infty\left(e^{-tH}f(x)-f(x)\right)~\frac{dt}{t^{1+\sigma}},
\end{equation}
where $v(x,t)=e^{-tH}f(x)$ is the solution of the heat-diffusion equation $\partial_tv+Hv=0$ in $\Real^n\times(0,\infty)$, with initial datum $v(x,0)=f(x)$ on $\Real^n$. In \cite{Stinga-Torrea} it is shown that
\begin{equation}\label{H sig f point}
H^\sigma f(x)=\int_{\Real^n}(f(x)-f(z))F_\sigma(x,z)~dz+f(x)B_\sigma(x),\qquad x\in\Real^n,~f\in\mathcal{S},
\end{equation}
where
\begin{equation}\label{FyB}
F_\sigma(x,z)=\frac{1}{-\Gamma(-\sigma)}\int_0^\infty G_t(x,z)~\frac{dt}{t^{1+\sigma}},\qquad B_\sigma(x)=\frac{1}{\Gamma(-\sigma)}\int_0^\infty\left[\int_{\Real^n}G_t(x,z)~dz-1\right]\frac{dt}{t^{1+\sigma}},
\end{equation}
and $G_t(x,z)$ is the kernel of the heat-difussion semigroup generated by $H$, see \eqref{heat int}.

Next we define the H\"{o}lder spaces in which the regularity properties of the operators will be considered.

\begin{defn}\label{calfa}
Let $0<\alpha\leq1$. A continuous function $u:\Real^n\to\Real$ belongs to the Hermite-H\"{o}lder space $C_H^{0,\alpha}$ associated to $H$, if there exists a constant $C$, depending only on $u$ and $\alpha$, such that
$$\abs{u(x_1)-u(x_2)}\leq C\abs{x_1-x_2}^\alpha,\quad\hbox{ and }\quad\abs{u(x)}\leq \frac{C}{(1+\abs{x})^\alpha},$$
for all $x_1,x_2,x\in\Real^n$. With the notation
$$[u]_{C^{0,\alpha}}=\sup_{\begin{subarray}{c}x_1,x_2\in\Real^n\\x_1\neq x_2\end{subarray}}\frac{\abs{u(x_1)-u(x_2)}}{\abs{x_1-x_2}^\alpha},\quad\hbox{ and }\quad[u]_{M^\alpha}=\sup_{x\in\Real^n}\abs{(1+\abs{x})^\alpha u(x)},$$
we define the norm in the spaces $C_H^{0,\alpha}$ to be $$\norm{u}_{C_H^{0,\alpha}}=[u]_{C^{0,\alpha}}+[u]_{M^\alpha}.$$
\end{defn}

When working with the harmonic oscillator $H$, some special first order partial differential operators are considered, see \eqref{derivatives}, which are the natural \textit{derivatives}. Then, in a usual way, the classes $C^{k,\alpha}_H$ can be defined, see Definition \ref{ckalfa}. We present now our first main result.

\

\noindent{\bf Theorem A.}
\textit{Let $\alpha\in(0,1]$ and $\sigma\in(0,1)$.
\begin{enumerate}[({A}1)]
   \item  Let $u\in C^{0,\alpha}_H$ and $2\sigma<\alpha$. Then $H^\sigma u\in C_H^{0,\alpha-2\sigma}$ and $$\norm{H^\sigma u}_{C_H^{0,\alpha-2\sigma}}\leq C\norm{u}_{C_H^{0,\alpha}}.$$
   \item Let $u\in C_H^{1,\alpha}$ and $2\sigma<\alpha$. Then $H^\sigma u\in C_H^{1,\alpha-2\sigma}$ and
        $$\norm{H^\sigma u}_{C_H^{1,\alpha-2\sigma}}\leq C\norm{u}_{C_H^{1,\alpha}}.$$
   \item Let $u\in C_H^{1,\alpha}$ and $2\sigma\geq\alpha$, with $\alpha-2\sigma +1\neq 0$. Then $H^\sigma u\in C_H^{0,\alpha-2\sigma+1}$ and
        $$\norm{H^\sigma u}_{C_H^{0,\alpha-2\sigma+1}}\leq C\norm{u}_{C_H^{1,\alpha}}.$$
   \item Let $u\in C_H^{k,\alpha}$ and assume that $k+\alpha-2\sigma$ is not an integer. Then $H^\sigma u\in C_H^{l,\beta}$ where $l$ is the integer part of $k+\alpha-2\sigma$ and $\beta=k+\alpha-2\sigma-l$.
\end{enumerate}}

\

The last Theorem can be interpreted as saying that the H\"{o}lder spaces $C^{k,\alpha}_H$ are the reasonable classes in order to obtain Schauder type estimates for $H^\sigma$. Indeed, if we define the negative powers of $H$, i.e. the fractional integral operators
$$H^{-\sigma}f(x)=\frac{1}{\Gamma(\sigma)}\int_0^\infty e^{-tH}f(x)~\frac{dt}{t^{1-\sigma}},\qquad0<\sigma\leq1,$$
then, we are able to prove our second main result:

\

\noindent{\bf Theorem B.}
\textit{Let $u\in C_H^{0,\alpha}$, for some $0<\alpha\leq1$, and $0<\sigma\leq1$.
\begin{enumerate}[({B}1)]
  \item If $\alpha+2\sigma\leq1$, then $H^{-\sigma}u\in C_H^{0,\alpha+2\sigma}$ and
  $$\norm{H^{-\sigma}u}_{C_H^{0,\alpha+2\sigma}}\leq C\norm{u}_{C_H^{0,\alpha}}.$$
  \item If $1<\alpha+2\sigma\leq2$, then $H^{-\sigma}u\in C_H^{1,\alpha+2\sigma-1}$ and
  $$\norm{H^{-\sigma}u}_{C_H^{1,\alpha+2\sigma-1}}\leq C\norm{u}_{C_H^{0,\alpha}}.$$
  \item If $2<\alpha+2\sigma\leq3$, then $H^{-\sigma}u\in C^{2,\alpha+2\sigma-2}_H$ and
  $$\norm{H^{-\sigma}u}_{C_H^{2,\alpha+2\sigma-2}}\leq C\norm{u}_{C_H^{0,\alpha}}.$$
\end{enumerate}}

\

It is worth noting that, if in Theorem B \textit{(B3)} we take $\sigma=1$, we get the Schauder estimate for the solution to $Hu=f$, in $\Real^n$, with $f\in C^{0,\alpha}_H$.

To prove the two main Theorems of this paper, we need a study of the first and second order Hermite-Riesz transforms, $\mathcal{R}_i$ and $\mathcal{R}_{ij}$, acting on the spaces $C^{0,\alpha}_H$, that we believe of independent interest. See the final part of Section \ref{Section:Operators} and Theorem \ref{Thm:R y Calfa}.

The first main task in our paper is to obtain explicit pointwise expressions for all the operators involved, when they are applied to functions belonging to the spaces $C^{k,\alpha}_H$, and the second one is to actually prove the regularity estimates. Section \ref{Section:Abstracta} contains two abstract Propositions dealing with these two aspects: Proposition \ref{Prop:tecnico} takes care of the pointwise formulas and Proposition \ref{Prop:tecnico2} contains a regularity result. We will apply, in a systematic way, both Propositions in order to reach our objectives: see Section \ref{Section:Operators} for all the pointwise formulas, and Section \ref{Section:Proofs} for the proofs of Theorems A and B. In Section \ref{Section:Tecnica} we collect all the computational Lemmas used in the previous sections.

In some recent papers, B. Bongioanni, E. Harboure and O. Salinas studied the boundedness of fractional integrals (see \cite{Bongioanni-Harboure-Salinas}) and Riesz transforms (see \cite{Bongioanni-Harboure-Salinas Riesz}), associated to a certain class of Schr\"odinger operators $\mathcal{L}=-\Delta+V$, in spaces of $BMO^\beta_{\mathcal{L}}$ type, $0\leq\beta<1$, using Harmonic Analysis techniques. In \cite{Bongioanni-Harboure-Salinas}, Proposition 4, they showed that the $BMO^\beta_{\mathcal{L}}$ spaces coincide with a H\"older type space $\Lambda^\beta_{\mathcal{L}}$, $0<\beta<1$, with equivalent norms. In the case $V=\abs{x}^2$, our space $C^{0,\beta}_H$ coincides with their space $\Lambda^\beta_H$, for $0<\beta<1$.

A natural question to think about is the possibility of getting (at least the local part of) our results by modifying, in an appropriate way, the kernel of the classical fractional Laplace operator. In our opinion our procedure is the natural one and we haven't found a smooth bridge to pass from one case to the other. Even more, some recent (local) results by  R. F. Bass in \cite{Bass} about \textit{stable-like operators} cannot be applied in our case because, clearly, his assumption on the kernel $A(x,h)$ (Assumption 1.1, 1) is not fulfilled by our kernel $F_\sigma(x,z)$, see Lemma \ref{Lem:F sig est} below. Moreover, he does not allow for $\alpha+\beta$ to be an integer (Assumption 1.1, 3), but we do. We want to complete the thought of this paragraph by establishing the parallelism with possible definitions of Sobolev spaces for the harmonic oscillator that were considered by R. Radha and S. Thangavelu in \cite{Radha-Thangavelu}, by S. Thangavelu in \cite{Thangavelu2} and also by B. Bongioanni and J. L. Torrea in  \cite{Bongioanni-Torrea}. In that case natural definitions of Sobolev spaces were given in order to get the results for the harmonic oscillator.

Throughout this paper, the letter $C$ denotes a positive constant that may change in each occurrence and it will depend on the parameters involved (whenever it is necessary, we point out this dependence with subscripts), and $\Gamma$ stands for the Gamma function. Recall that $\Gamma(-\sigma)=-\Gamma(1-\sigma)/\sigma$, $0<\sigma<1$. Without mention it, we will repeatedly apply the inequality $r^\eta e^{-r}\leq C_\eta e^{-r/2}$, $\eta\geq0$, $r>0$.

\section{Two abstract results}\label{Section:Abstracta}

\begin{prop}\label{Prop:tecnico}
Let $T$ be a bounded operator on $\mathcal{S}$, such that $\langle Tf,g\rangle=\langle f,Tg\rangle$, for all $f,g\in\mathcal{S}$. Assume that
$$Tf(x)=\int_{\Real^n}(f(x)-f(z))K(x,z)~dz+f(x)B(x),\qquad x\in\Real^n,$$
where the kernel $K$ verifies
\begin{equation}\label{K est}
\abs{K(x,z)}\leq\frac{C}{\abs{x-z}^{n+\gamma}}~e^{-\frac{\abs{x}\abs{x-z}}{C}}e^{-\frac{\abs{x-z}^2}{C}},\qquad x,z\in\Real^n,
\end{equation}
for some $-n\leq\gamma<1$, and $B$ is a continuous function with polynomial growth at infinity. Let $u\in C_H^{0,\gamma+\varepsilon}$, with $0<\gamma+\varepsilon\leq1$, $\varepsilon>0$. Then $Tu$ is well defined as a tempered distribution and it coincides with the continuous function
\begin{equation}\label{Tu}
Tu(x)=\int_{\Real^n}(u(x)-u(z))K(x,z)~dz+u(x)B(x),\qquad x\in\Real^n.
\end{equation}
\end{prop}

\begin{proof}
By \eqref{K est} and the smoothness of $u$, the integral in \eqref{Tu} is absolutely convergent. Since $B$ has polynomial growth at infinity, the right hand side of \eqref{Tu} defines a tempered distribution. Let us take $\frac{n}{\gamma+\varepsilon}<p<\infty$. Then, the finiteness of $[u]_{M^{\gamma+\varepsilon}}$ implies that $u\in L^p(\Real^n)$, and $Tu$ is well defined as a tempered distribution. Fix an arbitrary positive number $\eta$ and suppose that $R>0$. Let $f_j(x):=\zeta(x/j)\left(u\ast W_{1/j}\right)(x)$, $j\in\mathbb{N}$, where $W_t(z)=(4\pi t)^{-n/2}e^{-\abs{z}^2/(4t)}$ is the Gauss-Weierstrass kernel and $\zeta$ is a nonnegative smooth cutoff function (that is, $\zeta\in C_c^\infty(\Real^n)$, $0\leq\zeta\leq1$, $\zeta\equiv1$ in $B_1(0)$, $\zeta\equiv0$ in $B_2^c(0)$, and $\abs{\nabla\zeta}<C$ in $\Real^n$). Note that each $f_j$ belongs to $\mathcal{S}$. It is easy to check that the sequence $\set{f_j}_{j\in\mathbb{N}}$ converges to $u$ in $L^p(\Real^n)$ and uniformly in $B_R(x)$ for each $x\in\Real^n$, and $[f_j]_{C^{0,\gamma+\varepsilon}}\leq C\norm{u}_{C^{0,\gamma+\varepsilon}_H}=:M$. As $j\to\infty$, $Tf_j\to Tu$ in $\mathcal{S}'$. Since $B$ is a continuous function, $f_jB$ converges uniformly to $uB$ in $B_R(x_0)$, $x_0\in\Real^n$. There exists $0<\delta<R/2$ such that
$$M\int_{B_\delta(0)}\abs{z}^{\varepsilon-n}~dz\leq\frac{\eta}{3}.$$
For $x\in B_{R/2}(x_0)$, we write
$$\int_{\Real^n}(f_j(x)-f_j(z))K(x,z)~dz=\int_{B_\delta(x)}~+\int_{B_\delta^c(x)}~=I+II.$$
Then, by the choice of $\delta$,
$$\abs{I}+\abs{\int_{B_\delta(x)}(u(x)-u(z))K(x,z)~dz}\leq\frac{2}{3}~\eta.$$
We also have
$$\abs{II-\int_{B_\delta^c(x)}(u(x)-u(z))K(x,z)~dz}\leq C\abs{f_j(x)-u(x)}+C\left(\int_{B_\delta^c(x)}\abs{f_j(z)-u(z)}^p~dz\right)^{1/p}\leq\frac{\eta}{3},$$
for sufficiently large $j$, uniformly in $x\in B_{R/2}(x_0)$. Therefore,
$$\int_{\Real^n}(f_j(x)-f_j(z))K(x,z)~dz\rightrightarrows\int_{\Real^n}(u(x)-u(z))K(x,z)~dz,\quad j\to\infty,$$
in $B_{R/2}(x_0)$. Hence, by uniqueness of the limits, $Tu$ is a function that coincides with \eqref{Tu}, and it is a continuous function, because it is the uniform limit of continuous functions.
\end{proof}

\begin{rem}\label{Rem:K L1}
In the context of Proposition \ref{Prop:tecnico}, assume that, instead of having estimate \eqref{K est} on the kernel, we just know that $\abs{K(x,z)}\leq\Phi(x-z)$, where $\Phi\in L^{p'}(\Real^n)$, and $p'$ is the conjugate exponent of some $p$ such that $\frac{n}{\gamma+\varepsilon}<p<\infty$. Then, it is enough to take $u\in C^{0,\alpha}_H$, for some $0<\alpha\leq1$, to get the same conclusion, since the approximation procedure given in the proof above can also be applied in this situation.
\end{rem}

\begin{prop}\label{Prop:tecnico2}
Let $T$ be an operator satisfying the hypotheses of Proposition \ref{Prop:tecnico}, with $0\leq\gamma<1$ and $0<\gamma+\varepsilon\leq1$, for some $0<\varepsilon<1$. Assume that the kernel $K$ and the function $B$ also satisfy:
\begin{enumerate}[(a)]
  \item $\displaystyle \abs{K(x_1,z)-K(x_2,z)}\leq C\frac{\abs{x_1-x_2}}{\abs{x_2-z}^{n+1+\gamma}}~e^{-\frac{\abs{z}\abs{x_2-z}^2}{C}}e^{-\frac{\abs{x_2-z}^2}{C}}$, when $\abs{x_1-z}>2\abs{x_1-x_2}$.
  \item There exists a constant $C>0$ such that $\displaystyle\abs{\int_{\abs{x-z}>r}K(x,z)~dz}\leq Cr^{-\gamma}$, for all $x\in\Real^n$.
  \item For all $x\in\Real^n$, $\abs{B(x)}\leq C(1+\abs{x})^\gamma$, and $\nabla B\in L^\infty(\Real^n)$.
\end{enumerate}
Then $T$ maps $C^{0,\gamma+\varepsilon}_H$ into $C^{0,\varepsilon}_H$ continuously.
\end{prop}

\begin{proof}
Given $x_1,x_2\in\Real^n$, let $B=B(x_1,2\abs{x_1-x_2})$, $\widetilde{B}=B(x_2,4\abs{x_1-x_2})$ and $B'=B(x_2,\abs{x_1-x_2})$. We write
$$Tu(x)=\int_B(u(x)-u(z))K(x,z)dz+\int_{B^c}(u(x)-u(z))K(x,z)dz+u(x)B(x)=I(x)+II(x)+III(x).$$

By \eqref{K est} we have
\begin{align*}
  \abs{I(x_1)-I(x_2)} &\leq \int_B\abs{(u(x_1)-u(z))K(x_1,z)}~dz+\int_{\widetilde{B}}\abs{(u(x_2)-u(z))K(x_2,z)}~dz \\
   &\leq C[u]_{C^{0,{\gamma+\varepsilon}}}\left[\int_B\frac{\abs{x_1-z}^{\gamma+\varepsilon}}{\abs{x_1-z}^{n+\gamma}}~dz+ \int_{\widetilde{B}}\frac{\abs{x_2-z}^{\gamma+\varepsilon}}{\abs{x_2-z}^{n+\gamma}}~dz\right]=C[u]_{C^{0,{\gamma+\varepsilon}}}\abs{x_1-x_2}^\varepsilon.
\end{align*}

For the difference $II(x_1)-II(x_2)$, we add the term $\pm u(x_2)K(x_1,z)$ and we use the smoothness and cancelation properties of the kernel $K(x,z)$ (hypotheses \textit{(a)} and \textit{(b)}) to get:
\begin{align*}
  \abs{II(x_1)-II(x_2)} &\leq \int_{B^c}\abs{(u(x_2)-u(z))(K(x_1,z)-K(x_2,z))}dz+\abs{\int_{B^c}(u(x_1)-u(x_2))K(x_1,z)dz} \\
   &\leq C[u]_{C^{0,{\gamma+\varepsilon}}}\left[\int_{B^c}\abs{x_2-z}^{\gamma+\varepsilon}\frac{\abs{x_1-x_2}}{\abs{x_2-z}^{n+1+\gamma}}~dz+ \abs{x_1-x_2}^{\gamma+\varepsilon}\abs{\int_{B^c}K(x_1,z)dz}\right] \\
   &\leq C[u]_{C^{0,{\gamma+\varepsilon}}}\left[\int_{(B')^c}\frac{\abs{x_1-x_2}}{\abs{x_2-z}^{n+1-\varepsilon}}~dz+ \abs{x_1-x_2}^\varepsilon\right]=C[u]_{C^{0,{\gamma+\varepsilon}}}\abs{x_1-x_2}^\varepsilon.
\end{align*}

If $\abs{x_1-x_2}<\frac{1}{1+\abs{x_1}}$, by \textit{(c)},
\begin{align*}
  \frac{\abs{III(x_1)-III(x_2)}}{\abs{x_1-x_2}^\varepsilon} &\leq \frac{\abs{u(x_1)-u(x_2)}}{\abs{x_1-x_2}^{\gamma+\varepsilon}}\abs{B(x_1)}\abs{x_1-x_2}^\gamma+\abs{u(x_2)}\frac{\abs{B(x_1)-B(x_2)}}{\abs{x_1-x_2}^\varepsilon} \\
   &\leq C[u]_{C^{0,\gamma+\varepsilon}}+[u]_{M^{\gamma+\varepsilon}}\norm{\nabla B}_{L^\infty(\Real^n)}\abs{x_1-x_2}^{1-\varepsilon}\leq C\norm{u}_{C_H^{0,\gamma+\varepsilon}}.
\end{align*}
Assume that $\abs{x_1-x_2}\geq\frac{1}{1+\abs{x_1}}$. Then $1+\abs{x_1}\leq 1+\abs{x_2}+\abs{x_1-x_2}$, which implies
$$\frac{1+\abs{x_1}}{1+\abs{x_2}}\leq1+\frac{\abs{x_1-x_2}}{1+\abs{x_2}},\quad\hbox{and then,}\quad\frac{1}{1+\abs{x_2}}\leq\frac{1}{1+\abs{x_1}}+\frac{\abs{x_1-x_2}}{(1+\abs{x_1})(1+\abs{x_2})}\leq2\abs{x_1-x_2}.$$
With this and hypothesis \textit{(c)} we have
$$\frac{\abs{III(x_1)-III(x_2)}}{\abs{x_1-x_2}^\varepsilon}\leq\abs{u(x_1)B(x_1)}(1+\abs{x_1})^\varepsilon+\abs{u(x_2)B_\sigma(x_2)}2^\varepsilon(1+\abs{x_2})^\varepsilon
\leq C[u]_{M^{\gamma+\varepsilon}}.$$

Let us finally study the growth of $Tu(x)$. For the multiplicative term $uB$ we clearly have $\abs{u(x)B(x)}\leq C[u]_{M^{\gamma+\varepsilon}}(1+\abs{x})^{-\varepsilon}$. Consider next the integral part in the formula for $Tu(x)$, \eqref{Tu}. Since $Tu$ and $B$ are continuous functions, it is enough to consider $\abs{x}>2$. We write
$$\abs{\int_{\Real^n}(u(x)-u(z))K(x,z)~dz}=\abs{\left(\int_{\abs{x-z}<\frac{1}{1+\abs{x}}}+\int_{\abs{x-z}\geq\frac{1}{1+\abs{x}}}\right)~dz}.$$
On one hand,
$$\int_{\abs{x-z}<\frac{1}{1+\abs{x}}}\abs{u(x)-u(z)}\abs{K(x,z)}~dz\leq C[u]_{C^{0,{\gamma+\varepsilon}}}\int_{\abs{x-z}<\frac{1}{1+\abs{x}}}\frac{\abs{x-z}^{\gamma+\varepsilon}}{\abs{x-z}^{n+\gamma}}~dz=[u]_{C^{0,{\gamma+\varepsilon}}}\frac{C}{(1+\abs{x})^\varepsilon}.$$
On the other hand, by \textit{(b)},
$$\abs{u(x)}\abs{\int_{\abs{x-z}\geq\frac{1}{1+\abs{x}}}K(x,z)~dz}\leq \frac{[u]_{M^{\gamma+\varepsilon}}}{(1+\abs{x})^{\gamma+\varepsilon}}~C(1+\abs{x})^\gamma=[u]_{M^{\gamma+\varepsilon}}\frac{C}{(1+\abs{x})^\varepsilon}.$$
Since $\abs{x-z}\geq\frac{1}{1+\abs{x}}$ implies that $\frac{1}{1+\abs{z}}\leq2\abs{x-z}$, applying \eqref{K est} we get
\begin{align*}
  \int_{\abs{x-z}\geq\frac{1}{1+\abs{x}}}\abs{u(z)}\abs{K(x,z)}~dz &\leq C[u]_{M^{\gamma+\varepsilon}}\int_{\abs{x-z}\geq\frac{1}{1+\abs{x}}}\frac{1}{(1+\abs{z})^{\gamma+\varepsilon}}~\frac{e^{-\frac{\abs{x}\abs{x-z}}{C}}e^{-\frac{\abs{x-z}^2}{C}}}{\abs{x-z}^{n+\gamma}}~dz \\
   &\leq C[u]_{M^{\gamma+\varepsilon}}\int_{\abs{x-z}\geq\frac{1}{1+\abs{x}}}\abs{x-z}^{\gamma+\varepsilon}\frac{e^{-\frac{\abs{x}\abs{x-z}}{C}}e^{-\frac{\abs{x-z}^2}{C}}}{\abs{x-z}^{n+\gamma}}~dz \\
   &= C[u]_{M^{\gamma+\varepsilon}}\sum_{j=0}^\infty\int_{\abs{x-z}\sim\frac{2^j}{1+\abs{x}}}\frac{e^{-\frac{\abs{x}\abs{x-z}}{C}}e^{-\frac{\abs{x-z}^2}{C}}}{\abs{x-z}^{n-\varepsilon}}~dz \\
   &\leq [u]_{M^{\gamma+\varepsilon}}\frac{C}{(1+\abs{x})^\varepsilon}\sum_{j=0}^\infty2^{j\varepsilon}e^{-\frac{2^j}{C'}}=[u]_{M^{\gamma+\varepsilon}}\frac{C}{(1+\abs{x})^\varepsilon},
\end{align*}
where in the last line the constant $C'$ appearing in the exponential is independent of $x$ because $\abs{x}\abs{x-z}\sim2^j$. Therefore, by pasting the estimates above, the result is proved.
\end{proof}

\section{The operators}\label{Section:Operators}

In this section we give the pointwise definitions, in the class $C^{k,\alpha}_H$, of all the operators involved.

\section*{The heat-diffusion semigroup: $e^{-tH}$}

In our paper, the kernel of the heat-diffusion semigroup generated by $H$ will play an essential role. We shall need the pointwise formula for it.

Recall (see \cite{Thangavelu}) that the eigenfunctions of $H$ are the multi-dimensional Hermite functions $h_\nu(x)=e^{-\abs{x}^2/2}\Psi_\nu(x)$, $\nu=(\nu_1,\ldots,\nu_n)\in\mathbb{N}_0^n$, where $\Psi_\nu$ are the multi-dimensional Hermite polynomials, with positive eigenvalues: $Hh_\nu=(2\abs{\nu}+n)h_\nu,$ for all $\nu\in\mathbb{N}_0^n$, $\abs{\nu}=\nu_1+\cdots+\nu_n$. Moreover, $\overline{\spa\set{h_\nu:\nu\in\mathbb{N}_0^n}}=L^2(\Real^n)$. The heat-diffusion semigroup generated by $H$ is given as an integral operator: for $u\in\underset{1\leq p\leq\infty}{\bigcup}L^p(\Real^n)$,
\begin{multline}\label{heat int}
e^{-tH}u(x)=\int_{\Real^n}G_t(x,z)u(z)~dz \\
=\int_{\Real^n}\left[\sum_{j=0}^\infty e^{-t(2j+n)}\sum_{\abs{\nu}=j}h_\nu(x)h_\nu(z)\right]u(z)~dz =\int_{\Real^n}\frac{e^{-\left[\frac{1}{2}\abs{x-z}^2\coth2t+x\cdot z\tanh t\right]}}{(2\pi\sinh2t)^{n/2}}~u(z)~dz.
\end{multline}
Note that, for all $\nu\in\mathbb{N}_0^n$, $e^{-tH}h_\nu(x)=e^{-t(2\abs{\nu}+n)}h_\nu(x)$, $t\geq0$. With the following change of parameters (due to S. Meda)
\begin{equation}\label{Meda transf}
t=\frac{1}{2}\log\frac{1+s}{1-s},\qquad t\in(0,\infty),~s\in(0,1),
\end{equation}
the heat-diffusion kernel can be expressed as
$$G_{t(s)}(x,z)=\sum_{j=0}^\infty\left(\frac{1-s}{1+s}\right)^{j+n/2}\sum_{\abs{\nu}=j}h_\nu(x)h_\nu(z)=\left(\frac{1-s^2}{4\pi s}\right)^{n/2}e^{-\frac{1}{4}\left[s\abs{x+z}^2+\frac{1}{s}\abs{x-z}^2\right]},\quad s\in(0,1).$$

\section*{The fractional operators: $H^\sigma$ and $(H\pm2k)^\sigma$, $k\in\mathbb{N}$}

Let us first analyze the fractional harmonic oscillator $H^\sigma$. Let $\langle f,h_\nu\rangle=\int_{\Real^n}f(z)h_\nu(z)~dz$. If $f\in\mathcal{S}$, the Hermite series expansion $\sum_\nu\langle f,h_\nu\rangle h_\nu=\sum_{k=0}^\infty\sum_{\abs{\nu}=k}\langle f,h_\nu\rangle h_\nu$, converges to $f$ uniformly in $\Real^n$ (and also in $L^2(\Real^n)$), since $\norm{h_\nu}_{L^\infty(\Real^n)}\leq C$, for all $\nu\in\mathbb{N}_0^n$, and, for each $m\in\mathbb{N}$, we have $\abs{\langle f,h_\nu\rangle}\leq\norm{H^mf}_{L^2(\Real^n)}(2\abs{\nu}+n)^{-m}$. As $e^{-tH}f(x)=\sum_\nu e^{-t(2\abs{\nu}+n)}\langle f,h_\nu\rangle h_\nu$, from \eqref{H sig heat} we get $H^\sigma f=\sum_\nu(2\abs{\nu}+n)^\sigma\langle f,h_\nu\rangle h_\nu$, and the series converges uniformly in $\Real^n$. As a consequence of the last reasonings, $H^\sigma$ is a bounded operator in $\mathcal{S}$. Note that, by using Hermite series expansions, we can check that $\langle H^\sigma f,g\rangle=\langle f,H^\sigma g\rangle$, for all $f,g\in\mathcal{S}$, and $H^1f=Hf$, $H^0f=f$.

The proof of the identity \eqref{H sig f point} can be found in \cite{Stinga-Torrea} and we sketch it here for completeness. Since $e^{-tH}1(x)$ is not a constant function, we have
\begin{eqnarray*}
  \lefteqn{\frac{1}{\Gamma(-\sigma)}\int_0^\infty\left(e^{-tH}f(x)-f(x)\right)~\frac{dt}{t^{1+\sigma}}=\frac{1}{\Gamma(-\sigma)}\int_0^\infty\left(\int_{\Real^n}G_t(x,z)f(z)~dz-f(x)\right)~\frac{dt}{t^{1+\sigma}}} \\
   &=& \frac{1}{\Gamma(-\sigma)}\int_0^\infty\left[\int_{\Real^n}G_t(x,z)(f(z)-f(x))~dz+f(x)\left(\int_{\Real^n}G_t(x,z)~dz-1\right)\right]~\frac{dt}{t^{1+\sigma}} \\
   &=& \frac{1}{\Gamma(-\sigma)}\int_0^\infty\int_{\Real^n}G_t(x,z)(f(z)-f(x))~dz~\frac{dt}{t^{1+\sigma}}+f(x)\frac{1}{\Gamma(-\sigma)}\int_0^\infty\left(e^{-tH}1(x)-1\right)\frac{dt}{t^{1+\sigma}} \\
   &=& \int_{\Real^n}(f(x)-f(z))F_\sigma(x,z)~dz+f(x)B_\sigma(x).
\end{eqnarray*}
The subtle point in the calculations above is to justify the last equality. If $0<\sigma<1/2$, the last integral is absolutely convergent. In the case $1/2\leq\sigma<1$, a cancelation is involved (which is also exploited in the proof of Theorem \ref{Thm:H sig C alfa} below), that allows to show that the integral converges as a principal value.

As we said in the introduction, some type of \textit{derivatives} (first order partial differential operators) are usually considered when working with the operator $H$. Recall the factorization
$$H=\frac{1}{2}\sum_{i=1}^n(A_iA_{-i}+A_{-i}A_i),$$
where
\begin{equation}\label{derivatives}
A_i=\partial_{x_i}+x_i,\quad A_{-i}=A_i^\ast=-\partial_{x_i}+x_i,\qquad i=1,\ldots,n.
\end{equation}
In the Harmonic Analysis associated to $H$, the operators $A_i$, $1\leq\abs{i}\leq n$, play the role of the classical partial derivatives $\partial_{x_i}$ in the Euclidean Harmonic Analysis (see \cite{Thangavelu}, \cite{Bongioanni-Torrea}, \cite{Harboure-deRosa-Segovia-Torrea}, \cite{Stempak-Torrea Acta}). Now, it is natural to consider the classes of functions whose $k$-th \textit{derivatives} are in $C^{0,\alpha}_H$.

\begin{defn}\label{ckalfa}
For each $k\in\mathbb{N}$, we define the Hermite-H\"{o}lder space $C_H^{k,\alpha}$, $0<\alpha\leq1$, as the set of all  functions $u \in C^k(\Real^n)$ such that the following norm is finite:
$$\norm{u}_{C_H^{k,\alpha}}=[u]_{M^\alpha}+\sum_{\begin{subarray}{c}1\leq\abs{i_1},\ldots,\abs{i_m}\leq n\\1\leq m\leq k\end{subarray}}[A_{i_1}\cdots A_{i_m}u]_{M^\alpha}+\sum_{1\leq\abs{i_1},\ldots,\abs{i_k}\leq n}[A_{i_1}\cdots A_{i_k}u]_{C^{0,\alpha}}.$$
\end{defn}

We are ready to show that the pointwise formula for $H^\sigma u$, when $u$ belongs to the H\"older class $C^{k,\alpha}_H$, is the same as \eqref{H sig f point}.

\begin{thm}\label{Thm:H sig C alfa}
Let $0<\alpha\leq1$ and $0<\sigma<1$.
\begin{enumerate}[(1)]
   \item If $0<\alpha-2\sigma<1$ and $u\in C^{0,\alpha}_H$, then
   \begin{equation}\label{H sig u point}
   H^\sigma u(x)=\int_{\Real^n}(u(x)-u(z))F_\sigma(x,z)~dz+u(x)B_\sigma(x),\qquad x\in\Real^n,
   \end{equation}
   and the integral converges absolutely.
   \item If $-1<\alpha-2\sigma\leq0$ and $u\in C^{1,\alpha}_H$, then $H^\sigma u(x)$ is given by \eqref{H sig u point}, where the integral converges as a principal value.
   \item When $-2<\alpha-2\sigma\leq-1$, it is enough to take $u\in C^{1,1}_H$ to have the conclusion of (2).
\end{enumerate}
In the three cases, $H^\sigma u\in C(\Real^n)$.
\end{thm}

\begin{proof}
If $0<\alpha-2\sigma<1$, then $\sigma<1/2$. The properties of $F_\sigma$ and $B_\sigma$ established in Lemmas \ref{Lem:F sig est} and \ref{Lem:B sig est} (see Section \ref{Section:Tecnica}), allow us to apply Proposition \ref{Prop:tecnico}, with $K(x,z)=F_\sigma(x,z)$, $B=B_\sigma$ and $\gamma=2\sigma<1$, to get \textit{(1)}.

Under the hypotheses of (2), we will take advantage of a cancelation to show that the integral in \eqref{H sig u point} is well defined. Suppose that $\delta>0$. By Lemma \ref{Lem:F sig est},
$$\int_{\abs{x-z}\geq\delta}\abs{u(x)-u(z)}F_\sigma(x,z)~dz\leq C_\delta\norm{u}_{L^\infty(\Real^n)}\int_{\abs{x-z}\geq\delta}e^{-\frac{\abs{x-z}^2}{C}}~dz<\infty.$$
For $\rho\in\Real$, the change of parameters \eqref{Meda transf} produces
$$\frac{dt}{t^{1+\rho}}=d\mu_\rho(s):=\frac{ds}{(1-s^2)\left(\frac{1}{2}\log\frac{1+s}{1-s}\right)^{1+\rho}},\qquad t\in(0,\infty),~s\in(0,1),$$
so that,
\begin{equation}\label{F sig s}
F_\sigma(x,z)=\frac{1}{-\Gamma(-\sigma)}\int_0^1\left(\frac{1-s^2}{4\pi s}\right)^{n/2}e^{-\frac{1}{4}\left[s\abs{x+z}^2+\frac{1}{s}\abs{x-z}^2\right]}d\mu_\sigma(s),
\end{equation}
which, up to the multiplicative constant $1/(-\Gamma(-\sigma))$, gives,
\begin{align*}
  I &= \int_{\abs{z}<\delta}(u(x)-u(x-z))F_\sigma(x,x-z)~dz \\
   &= \int_0^\delta r^{n-1}\int_{\abs{z'}=1}(u(x)-u(x-rz'))\int_0^1\left(\frac{1-s^2}{4\pi s}\right)^{n/2}e^{-\frac{1}{4}\left[s\abs{2x-rz'}^2+\frac{r^2}{s}\right]}~d\mu_\sigma(s)~dS(z')~dr.
\end{align*}
By the smoothness of $u$, $u(x)-u(x-rz')=\nabla u(x)(rz')+R_1u(x,rz')$, with $\abs{R_1u(x,rz')}\leq[\nabla u]_{C^{0,\alpha}}r^{1+\alpha}$. We apply the Mean Value Theorem to the function $\psi(x)=\psi_{s,r}(x)=e^{-\frac{1}{4}\left[s\abs{x}^2+\frac{r^2}{s}\right]}$, to see that
$e^{-\frac{1}{4}\left[s\abs{2x-rz'}^2+\frac{r^2}{s}\right]}=e^{-\frac{1}{4}\left[s\abs{2x}^2+\frac{r^2}{s}\right]}+R_0\psi(x,rz')$, with $\abs{R_0\psi(x,rz')}\leq Cs^{1/2}re^{-r^2/(8s)}$. Therefore,
\begin{align*}
  I &= \int_0^\delta r^{n-1}\int_{\abs{z'}=1}\nabla u(x)(rz')\int_0^1\left(\frac{1-s^2}{4\pi s}\right)^{n/2}R_0\psi(x,rz')~d\mu_\sigma(s)~dS(z')~dr \\
   &\quad +\int_0^\delta r^{n-1}\int_{\abs{z'}=1}R_1u(x,rz')\int_0^1\left(\frac{1-s^2}{4\pi s}\right)^{n/2}\left(e^{-\frac{1}{4}\left[s\abs{2x}^2+\frac{r^2}{s}\right]}+R_0\psi(x,rz')\right)d\mu_\sigma(s)dS(z')dr \\ &=: I_1+I_2.
\end{align*}
Note that
\begin{equation}\label{est dmu rho}
d\mu_\rho(s)\sim\frac{ds}{s^{1+\rho}},~s\sim 0,\qquad d\mu_\rho(s)\sim\frac{ds}{(1-s)(-\log(1-s))^{1+\rho}},~s\sim1.
\end{equation}
With the estimates on $R_1u$ and $R_0\psi$ given above and \eqref{est dmu rho}, we obtain
$$\abs{I_1}\leq C\abs{\nabla u(x)}\int_0^\delta r^{n+1}\int_0^1\left(\frac{1-s}{s}\right)^{n/2}s^{1/2}e^{-\frac{r^2}{8s}}d\mu_\sigma(s)~dr\leq C\abs{\nabla u(x)}\int_0^\delta\frac{r^{n+1}}{r^{n-1+2\sigma}}~dr=C\delta^{3-2\sigma},$$
and
$$\abs{I_2}\leq C[\nabla u]_{C^{0,\alpha}}\int_0^\delta r^{n+\alpha}\int_0^1\left(\frac{1-s}{s}\right)^{n/2}e^{-\frac{r^2}{4s}}~d\mu_\sigma(s)~dr\leq C[\nabla u]_{C^{0,\alpha}}\int_0^\delta\frac{r^{n+\alpha}}{r^{n+2\sigma}}~dr=C\delta^{\alpha-2\sigma+1}.$$
Thus, the integral in \eqref{H sig u point} converges as a principal value. The same happens if we take $u\in C^{1,1}_H$: we repeat the argument above, but applying in $I_2$ the estimate $\abs{R_1u(x,rz')}\leq[\nabla u]_{C^{0,1}}\cdot r^2$.

To obtain the conclusions of \textit{(2)} and \textit{(3)}, we note that the approximation procedure used in the proof of Proposition \ref{Prop:tecnico} can be applied here (with the estimate $[\nabla f_j]_{C^{0,\alpha}}\leq C\norm{u}_{C^{1,\alpha}_H}=M$).
\end{proof}

\begin{rem}
As in \cite{Stinga-Torrea} (and \cite{Silvestre Thesis,Silvestre CPAM} for $(-\Delta)^\sigma$), some easy maximum and comparison principles can be derived from Theorem \ref{Thm:H sig C alfa}. For instance, take $0<\alpha\leq1$, $0<\sigma<1$, and $u,v$ in the class $C_H^{0,\alpha}$ (or $C_H^{1,\alpha}$, depending on the value of $\alpha-2\sigma$), such that $u\geq v$ in $\Real^n$, with $u(x_0)=v(x_0)$ for some $x_0\in\Real^n$. Then $H^\sigma u(x_0)\leq H^\sigma v(x_0)$. Moreover, $H^\sigma u(x_0)=H^\sigma v(x_0)$ only when $u\equiv v$.
\end{rem}

In order to prove the regularity estimates for $H^\sigma$, we will have to work with the derivatives of $H^\sigma$, that is, with operators of the type $A_iH^\sigma$, $1\leq\abs{i}\leq n$. We recall that, for all $\nu\in\mathbb{N}_0^n$, we have
$$A_ih_\nu=(2\nu_i)^{1/2}h_{\nu-e_i},\quad A_{-i}h_\nu=(2\nu_i+2)^{1/2}h_{\nu+e_i},\qquad 1\leq i\leq n,$$
where $e_i$ is the $i$-th coordinate vector in $\mathbb{N}_0^n$. Then, for all $f\in \mathcal{S}$ and $1\leq i\leq n$,
$$A_if=\sum_\nu(2\nu_i)^{1/2}\langle f,h_\nu\rangle h_{\nu-e_i},\quad A_{-i}f=\sum_\nu(2\nu_i+2)^{1/2}\langle f,h_\nu\rangle h_{\nu+e_i},$$
and both series converge uniformly in $\Real^n$.

\begin{rem}\label{Rem:H y Ai}
Let $b\in\Real$. Then, by using Hermite series expansions, it is easy to check that for all $f\in\mathcal{S}$ and $1\leq i\leq n$, we have
\begin{align*}
A_iH^bf &= (H+2)^bA_if, & H^bA_if &= A_i(H-2)^bf, \\
A_{-i}H^bf &= (H-2)^bA_{-i}f, & H^bA_{-i}f &= A_{-i}(H+2)^bf,
\end{align*}
where we defined $(H\pm2)^bh_\nu:=(2\abs{\nu}+n\pm2)^bh_\nu$.
\end{rem}

Consequently, we need to study the operators $(H\pm2k)^\sigma$, $k\in\mathbb{N}$.

Let us start with $(H+2k)^\sigma$, $k$ a positive integer. For $f\in\mathcal{S}$ and $k\in\mathbb{N}$ we define
$$(H+2k)^\sigma f(x)=\sum_\nu(2\abs{\nu}+n+2k)^\sigma\langle f,h_\nu\rangle h_\nu(x),\qquad x\in\Real^n.$$
The series above converges in $L^2(\Real^n)$ and uniformly in $\Real^n$, it defines a Schwartz's class function, and
$$(H+2k)^\sigma f(x)=\frac{1}{\Gamma(-\sigma)}\int_0^\infty\left(e^{-2kt}e^{-tH}f(x)-f(x)\right)~\frac{dt}{t^{1+\sigma}},\qquad x\in\Real^n.$$

By using Lemmas \ref{Lem:F sig est} and \ref{Lem:B sig est} stated in Section \ref{Section:Tecnica}, the following Theorem can be proved in a parallel way to Theorem \ref{Thm:H sig C alfa}. We leave the details to the interested reader.

\begin{thm}\label{Thm:H+2k}
Let $u$ be as in Theorem \ref{Thm:H sig C alfa}. Then $(H+2k)^\sigma u\in\mathcal{S}'\cap C(\Real^n)$, and
$$(H+2k)^\sigma u(x)=\int_{\Real^n}(u(x)-u(z))F_{2k,\sigma}(x,z)~dz+u(x)B_{2k,\sigma}(x),\qquad x\in\Real^n,$$
where
$$F_{2k,\sigma}(x,z)=\frac{1}{-\Gamma(-\sigma)}\int_0^\infty e^{-2kt}G_t(x,z)~\frac{dt}{t^{1+\sigma}}=\frac{1}{-\Gamma(-\sigma)}\int_0^1\left(\frac{1-s}{1+s}\right)^kG_{t(s)}(x,z)~d\mu_\sigma(s),$$
and
$$B_{2k,\sigma}(x)=\frac{1}{\Gamma(-\sigma)}\int_0^1\left[\left(\frac{1-s}{1+s}\right)^k\left(\frac{1-s^2}{2\pi(1+s^2)}\right)^{n/2}e^{-\frac{s}{1+s^2}\abs{x}^2}-1\right]~d\mu_\sigma(s).$$
\end{thm}

Consider next the operators $(H-2k)^\sigma$, $k\in\mathbb{N}$. We say that a function $f\in\mathcal{S}$ belongs to the space $\mathcal{S}_k$ if
$$\int_{\Real^n}f(z)h_\nu(z)~dz=0,\quad\hbox{ for all }\nu\in\mathbb{N}_0^n\hbox{ such that }\abs{\nu}<k.$$
For $f\in\mathcal{S}_k$, we define
$$(H-2k)^\sigma f(x)=\sum_{\abs{\nu}\geq k}(2\abs{\nu}+n-2k)^\sigma\langle f,h_\nu\rangle h_\nu(x).$$
Note that, on $\mathcal{S}_k$, the operator $(H-2k)^\sigma$ is positive. Let
$$\phi_{2k}(x)=\phi_{2k}(x,z,s)=\left[\sum_{j=0}^{k-1}\left(\frac{1-s}{1+s}\right)^{j+n/2}\sum_{\abs{\nu}=j}h_\nu(x)h_\nu(z)\right]\chi_{(1/2,1)}(s),$$
the sum of the first $(k-1)$-terms of the series defining $G_{t(s)}(x,z)$, for $s\in(1/2,1)$. Then, the heat-diffusion semigroup generated by $H-2k$:
$$e^{-t(H-2k)}f(x)=\int_{\Real^n}e^{2kt}G_t(x,z)f(z)~dz=\int_{\Real^n}\left(\frac{1+s}{1-s}\right)^kG_{t(s)}(x,z)f(z)~dz=e^{-t(s)(H-2k)}f(x),$$
can be written as
$$e^{-t(s)(H-2k)}f(x)=\int_{\Real^n}\left(\frac{1+s}{1-s}\right)^k\left[G_{t(s)}(x,z)-\phi_{2k}(x,z,s)\right]f(z)~dz,\qquad f\in\mathcal{S}_k.$$
Moreover,
\begin{align*}
  (H-2k)^\sigma f(x) &= \frac{1}{\Gamma(-\sigma)}\int_0^\infty\left(e^{-t(H-2k)}f(x)-f(x)\right)~\frac{dt}{t^{1+\sigma}} \\
   &= \frac{1}{\Gamma(-\sigma)}\int_0^1\left(e^{-t(s)(H-2k)}f(x)-f(x)\right)~d\mu_\sigma(s).
\end{align*}
The following idea is taken from \cite{Harboure-deRosa-Segovia-Torrea}. By the $n$-dimensional Mehler's formula (see \cite[p.~6]{Thangavelu}),
\begin{equation}\label{M_r}
M_r(x,z):=\sum_{j=0}^\infty r^j\sum_{\abs{\nu}=j}h_\nu(x)h_\nu(z)=\frac{1}{\pi^{n/2}(1-r^2)^{n/2}}~e^{-\frac{1}{4}\left[\frac{1-r}{1+r}\abs{x+z}^2+\frac{1+r}{1-r}\abs{x-z}^2\right]},\qquad r\in(0,1).
\end{equation}
Then, for all $r\in(0,1/3)$,
$$\abs{\frac{d^k}{dr^k}~M_r(x,z)}\leq C\left(1+\abs{x+z}^2+\abs{x-z}^2\right)^ke^{-\frac{1}{4}\left[\frac{1-r}{1+r}\abs{x+z}^2+\frac{1+r}{1-r}\abs{x-z}^2\right]}\leq Ce^{-\frac{\abs{x}\abs{x-z}}{C}}e^{-\frac{\abs{x-z}^2}{C}},$$
where in the second inequality we applied Lemma \ref{Lem:exp est} of Section \ref{Section:Tecnica}, with $s=\frac{1-r}{1+r}$. Thus, by Taylor's formula,
$$\abs{\sum_{j=k}^\infty r^j\sum_{\abs{\nu}=j}h_\nu(x)h_\nu(z)}\leq Cr^ke^{-\frac{\abs{x}\abs{x-z}}{C}}e^{-\frac{\abs{x-z}^2}{C}},\qquad r\in(0,1/3).$$
Therefore, letting $r=\frac{1-s}{1+s}$ above, we obtain
\begin{align}
  \label{dif1} \abs{G_{t(s)}(x,z)-\phi_{2k}(x,z,s)} &= \abs{\sum_{j=k}^\infty\left(\frac{1-s}{1+s}\right)^{j+n/2}\sum_{\abs{\nu}=j}h_\nu(x)h_\nu(z)} \\
  \label{dif2} &\leq C\left(\frac{1-s}{1+s}\right)^{k+n/2}e^{-\frac{\abs{x}\abs{x-z}}{C}}e^{-\frac{\abs{x-z}^2}{C}},\qquad\hbox{ for all }s\in(1/2,1).
\end{align}

If $u\in C^{k,\alpha}_H$, then we have
$$\int_{\Real^n}A_{-i_1}\cdots A_{-i_k}u(x)h_\nu(x)~dx=0,\qquad1\leq i_1,\ldots,i_k\leq n,~\abs{\nu}<k.$$

\begin{thm}\label{Thm:H-2k}
Let $0<\alpha\leq1$ and $0<\sigma<1$. Assume that $0<\alpha-2\sigma<1$ and take $u\in C^{k,\alpha}_H$. If $v(x)=(A_{-i_1}\cdots A_{-i_k}u)(x)$, $1\leq i_1,\ldots,i_k\leq n$, then $A_{-i_1}\cdots A_{-i_k}H^\sigma u\in\mathcal{S}'\cap C(\Real^n)$, and, for all $x\in\Real^n$,
\begin{equation}\label{H-2k def}
A_{-i_1}\cdots A_{-i_k}H^\sigma u(x)=(H-2k)^\sigma v(x)=\int_{\Real^n}(v(x)-v(z))F_{-2k,\sigma}(x,z)~dz+v(x)B_{-2k,\sigma}(x),
\end{equation}
where
$$F_{-2k,\sigma}(x,z)=\frac{1}{-\Gamma(-\sigma)}\int_0^1\left(\frac{1+s}{1-s}\right)^k\left[G_{t(s)}(x,z)-\phi_{2k}(x,z,s)\right]~d\mu_\sigma(s),$$
and
$$B_{-2k,\sigma}(x)=\frac{1}{\Gamma(-\sigma)}\int_0^1\left[\left(\frac{1+s}{1-s}\right)^k\int_{\Real^n}\left[G_{t(s)}(x,z)-\phi_{2k}(x,z,s)\right]~dz-1\right]~d\mu_\sigma(s).$$
The integral in \eqref{H-2k def} is absolutely convergent.
\end{thm}

\begin{proof}
Even if we have good estimates for $F_{-2k,\sigma}$ and $B_{-2k,\sigma}$ (see Lemmas \ref{Lem:F sig est} and \ref{Lem:B sig est}), we can not apply directly Proposition \ref{Prop:tecnico} here because the test space for $(H-2k)^\sigma$ is not $\mathcal{S}$, but $\mathcal{S}_k$. Nevertheless, the same ideas will work. Indeed, using Lemmas \ref{Lem:F sig est} and \ref{Lem:B sig est}, it can be checked that the conclusion is valid when $u$ is a Schwartz's class function (and then $v\in\mathcal{S}_k$), and, for the general result, we can apply the approximation procedure given in the proof of Proposition \ref{Prop:tecnico}, noting that $(A_{-i_1}\cdots A_{-i_k}f_j)(x)$ can be used to approximate $v(x)$.
\end{proof}

\section*{The fractional integral: $H^{-\sigma}$}

For $f\in\mathcal{S}$, the fractional integral $H^{-\sigma}f$, $0<\sigma\leq1$, is given by
$$H^{-\sigma}f(x)=\frac{1}{\Gamma(\sigma)}\int_0^\infty e^{-tH}f(x)~\frac{dt}{t^{1-\sigma}}=\sum_{\nu}\frac{1}{(2\abs{\nu}+n)^\sigma}~\langle f,h_\nu\rangle h_\nu(x),$$
and $H^{-\sigma}$ is a continuous and symmetric operator in $\mathcal{S}$. Moreover, $H^{-\sigma}f=(H^\sigma)^{-1}f$, $f\in\mathcal{S}$. By writing down the expression of the heat-diffusion semigroup, and applying Fubini's Theorem above, we get
$$H^{-\sigma}f(x)=\int_{\Real^n}\left[\frac{1}{\Gamma(\sigma)}\int_0^\infty G_t(x,z)~\frac{dt}{t^{1-\sigma}}\right]f(z)~dz=\int_{\Real^n}F_{-\sigma}(x,z)f(z)~dz.$$
In \cite{Bongioanni-Torrea} it is shown that the definition of $H^{-\sigma}$ extends to $f\in L^p(\Real^n)$, $1\leq p\leq\infty$, via the previous integral formula.

Observe that, for $f\in \mathcal{S}$, we have
$$H^{-\sigma}f(x)=\int_{\Real^n}(f(z)-f(x))F_{-\sigma}(x,z)~dz+f(x)H^{-\sigma}1(x),\qquad\hbox{for all }x\in\Real^n,$$
where
\begin{equation}\label{H-sig1}
H^{-\sigma}1(x)=\frac{1}{\Gamma(\sigma)}\int_0^\infty e^{-tH}1(x)~\frac{dt}{t^{1-\sigma}}=\int_{\Real^n}F_{-\sigma}(x,z)~dz.
\end{equation}
The next Theorem shows that the operator $H^{-\sigma}$ can be defined in $C^{0,\alpha}_H$ precisely by this formula.

\begin{thm}
For $u\in C^{0,\alpha}_H$, $0<\alpha\leq1$, and $0<\sigma\leq1$, $H^{-\sigma}u\in\mathcal{S}'\cap C(\Real^n)$, and
$$H^{-\sigma}u(x)=\int_{\Real^n}(u(z)-u(x))F_{-\sigma}(x,z)~dz+u(x)H^{-\sigma}1(x),\qquad x\in\Real^n.$$
\end{thm}

\begin{proof}
In Lemmas \ref{Lem:F-sig est} and \ref{Lem:H-sig 1} we collect the properties of the kernel $F_{-\sigma}(x,z)$ and the function $H^{-\sigma}1(x)$. When $n>2\sigma$, an application of Proposition \ref{Prop:tecnico} with $\gamma=-2\sigma$ and $\varepsilon=\alpha+2\sigma$ gives the result. For the case $n\leq2\sigma$, we use Remark \ref{Rem:K L1}.
\end{proof}

We shall also need to work with the derivatives of $H^{-\sigma}u$. The following Theorem gives the pointwise formula that will be used along the paper.

\begin{thm}\label{Thm:Ai H-sig}
Take $0<\alpha\leq1$ and $0<\sigma\leq1$, such that $\alpha+2\sigma>1$. If $u\in C^{0,\alpha}_H$ then, for each $1\leq\abs{i}\leq n$, we have $A_iH^{-\sigma}u\in\mathcal{S}'\cap C(\Real^n)$, and
$$A_iH^{-\sigma}u(x)=\int_{\Real^n}(u(z)-u(x))A_iF_{-\sigma}(x,z)~dz+u(x)A_iH^{-\sigma}1(x),\qquad x\in\Real^n.$$
\end{thm}

\begin{proof}
Let us first prove the result when $u=f\in\mathcal{S}$. It is enough to consider $1\leq i\leq n$. We have
$$A_iH^{-\sigma}f(x)=A_i\int_{\Real^n}(f(z)-f(x))F_{-\sigma}(x,z)~dz+\partial_{x_i}f(x)H^{-\sigma}1(x)+f(x)A_iH^{-\sigma}1(x).$$
We want to put the $A_i$ inside the integral. In order to do that, we apply a classical approximation argument given in the proof of Lemma 4.1 of \cite{Gilbarg-Trudinger}, that we sketch here. By estimate \eqref{Ai F-sig est}, Lemma \ref{Lem:H-sig 1} (see Section \ref{Section:Tecnica}), and the fact that $\alpha+2\sigma>1$, the function
$$g(x)=\int_{\Real^n}\partial_{x_i}\left[(f(z)-f(x))F_{-\sigma}(x,z)\right]~dz= \int_{\Real^n}(f(z)-f(x))\partial_{x_i}F_{-\sigma}(x,z)~dz-\partial_{x_i}f(x)H^{-\sigma}1(x),$$
is well defined. Fix a function $\phi\in C^1(\Real)$ satisfying $0\leq\phi\leq1$, $\phi(t)=0$ for $t\leq1$, $\phi(t)=1$ for $t\geq2$, and $0\leq\phi'\leq2$. Define, for $0<\varepsilon<1/2$,
$$h_\varepsilon(x)=\int_{\Real^n}(f(z)-f(x))F_{-\sigma}(x,z)\phi\left(\varepsilon^{-1}\abs{x-z}\right)~dz.$$
Then estimate \eqref{F-sig est} implies that, as $\varepsilon\to0$, $h_\varepsilon(x)$ converges uniformly in $\Real^n$ to
$$\int_{\Real^n}(f(z)-f(x))F_{-\sigma}(x,z)~dz.$$
Moreover, $h_\varepsilon\in C^1(\Real^n)$, and, again by \eqref{F-sig est} and \eqref{Ai F-sig est},
\begin{align*}
  \abs{g(x)-\partial_{x_i}h_\varepsilon(x)} &\leq \int_{\Real^n}\abs{\partial_{x_i}\left[(f(z)-f(x))F_{-\sigma}(x,z)\left(1-\phi\left(\varepsilon^{-1}\abs{x-z}\right)\right)\right]}~dz \\
   &\leq C_f\int_{\abs{x-z}<2\varepsilon}\left[F_{-\sigma}(x,z)+\abs{x-z}\abs{\nabla_xF_{-\sigma}(x,z)}+\abs{x-z}F_{-\sigma}(x,z)\frac{1}{\varepsilon}\right]~dz \\
   &\leq C_f\Phi_{n,\sigma}(\varepsilon),
\end{align*}
where $\Phi_{n,\sigma}(\varepsilon)\to0$, as $\varepsilon\to0$, uniformly in $x\in\Real^n$. Thus,
$$\partial_{x_i}\int_{\Real^n}(f(z)-f(x))F_{-\sigma}(x,z)~dz=\int_{\Real^n}(f(z)-f(x))\partial_{x_i}F_{-\sigma}(x,z)~dz-\partial_{x_i}f(x)H^{-\sigma}1(x),$$
and the Theorem is valid when $u$ is a Schwartz function. For the general case, $u\in C^{0,\alpha}_H$, we argue as follows. If $n>2\sigma-1$, then, by \eqref{Ai F-sig est}, we can apply Proposition \ref{Prop:tecnico} with $\gamma=1-2\sigma$ and $\varepsilon=\alpha+2\sigma-1$, and, if $n=2\sigma-1$, we can use Remark \ref{Rem:K L1}.
\end{proof}

\section*{The Hermite-Riesz transforms: $\mathcal{R}_i$ and $\mathcal{R}_{ij}$}

The first order Hermite-Riesz transforms are given by
$$\mathcal{R}_i=A_iH^{-1/2},\qquad1\leq\abs{i}\leq n.$$
These operators where first introduced and studied by Thangavelu \cite{Thangavelu}. The second order Hermite-Riesz transforms are (see \cite{Harboure-deRosa-Segovia-Torrea}, \cite{Stempak-Torrea Acta})
$$\mathcal{R}_{ij}=A_iA_jH^{-1},\qquad1\leq\abs{i},\abs{j}\leq n.$$
Using Hermite series expansions it is easy to check that the first and second order Hermite-Riesz transforms are symmetric operators in $\mathcal{S}$ and that they map $\mathcal{S}$ into $\mathcal{S}$ continuously.

Taking $\sigma=1/2$ in Theorem \ref{Thm:Ai H-sig} we get:

\begin{thm}\label{Thm:R1 point}
If $u\in C^{0,\alpha}_H$, $0<\alpha\leq1$, then, for all $1\leq\abs{i}\leq n$, we have $\mathcal{R}_iu\in\mathcal{S}'\cap C(\Real^n)$, and
$$\mathcal{R}_iu(x)=\int_{\Real^n}(u(z)-u(x))A_iF_{-1/2}(x,z)~dz+u(x)A_iH^{-1/2}1(x),\qquad x\in\Real^n.$$
\end{thm}

By using the properties of the kernel of the second order Riesz transform $R_{ij}(x,z)=A_iA_jF_{-1}(x,z)$ (see Lemma \ref{Lem:Rij est} in Section \ref{Section:Tecnica}), it is easy to get a pointwise description of $\mathcal{R}_{ij}f$, $f\in\mathcal{S}$. Hence, we can use Proposition \ref{Prop:tecnico}, with $\gamma=0$ and $\varepsilon=\alpha$, to have the following Theorem.

\begin{thm}
If $u\in C^{0,\alpha}_H$, $0<\alpha\leq1$, then, for all $1\leq\abs{i},\abs{j}\leq n$, we have $\mathcal{R}_{ij}u\in\mathcal{S}'\cap C(\Real^n)$, and
$$\mathcal{R}_{ij}u(x)=\int_{\Real^n}(u(z)-u(x))R_{ij}(x,z)~dz+u(x)A_iA_jH^{-1}1(x),\qquad x\in\Real^n.$$
\end{thm}

\section{Proofs of the main results}\label{Section:Proofs}

\subsection{Regularity properties of the Hermite-Riesz transforms}

As we already said in the Introduction, a study of the action of the Hermite-Riesz transforms in the H\"older spaces $C^{k,\alpha}_H$ is needed.

\begin{thm}\label{Thm:R y Calfa}
The Hermite-Riesz transforms $\mathcal{R}_i$ and $\mathcal{R}_{ij}$, $1\leq\abs{i},\abs{j}\leq n$, are bounded operators on the spaces $C^{0,\alpha}_H$: if $u\in C^{0,\alpha}_H$, for some $0<\alpha<1$, then $\mathcal{R}_iu,\mathcal{R}_{ij}u\in C^{0,\alpha}_H$, and
$$\norm{\mathcal{R}_iu}_{C^{0,\alpha}_H}+\norm{\mathcal{R}_{ij}u}_{C^{0,\alpha}_H}\leq C\norm{u}_{C^{0,\alpha}_H}.$$
\end{thm}

\begin{proof}
By Lemmas \ref{Lem:F-sig est}, \ref{Lem:H-sig 1}, and \ref{Lem:Torrea} of Section \ref{Section:Tecnica}, and Theorem \ref{Thm:R1 point}, the result for $\mathcal{R}_i$ can be deduced applying Proposition \ref{Prop:tecnico2}, with $\gamma=0$ and $\varepsilon=\alpha$.

Let us consider the operator $\mathcal{R}_{ij}$, for some $j\in\set{1,\ldots,n}$. Then, by Remark \ref{Rem:H y Ai}, we have
$$\mathcal{R}_{ij}=A_iA_jH^{-1}=A_i\left(A_jH^{-1/2}\right)H^{-1/2}=A_i\left[(H+2)^{-1/2}A_j\right]H^{-1/2}=A_i(H+2)^{-1/2}\circ\mathcal{R}_j.$$
Therefore, it is enough to prove that $A_i(H+2)^{-1/2}$ is a continuous operator on $C^{0,\alpha}_H$. When $f\in\mathcal{S}$, we can write
$$A_i(H+2)^{-1/2}f(x)=\int_{\Real^n}(f(z)-f(x))A_iF_{2,-1/2}(x,z)~dz+f(x)A_i(H+2)^{-1/2}1(x),$$
where
$$F_{2,-1/2}(x,z)=\frac{1}{\Gamma(1/2)}\int_0^1\left(\frac{1-s}{1+s}\right)G_{t(s)}(x,z)d\mu_{-1/2}(s),\quad(H+2)^{-1/2}1(x)=\int_{\Real^n}F_{2,-1/2}(x,z)dz.$$
Following the proof of Lemmas \ref{Lem:F-sig est} and \ref{Lem:H-sig 1} given in Section \ref{Section:Tecnica}, it can be checked that the kernel $A_iF_{2,-1/2}(x,z)$ and the function $(H+2)^{-1/2}1(x)$ share the same size and smoothness properties than the kernel $A_iF_{-1/2}(x,z)$ and the function $H^{-1/2}1(x)$ stated in the mentioned Lemmas (details are left to the reader). Thus, as a consequence of the results of Section \ref{Section:Abstracta}, $A_i(H+2)^{-1/2}:C^{0,\alpha}_H\to C^{0,\alpha}_H$ continuously. Therefore $\mathcal{R}_{ij}$ is a bounded operator on $C^{0,\alpha}_H$, when $j\in\set{1,\ldots,n}$.

Note that
$$\mathcal{R}_{ij}=\partial^2_{x_i,x_j}H^{-1}+x_j\partial_{x_i}H^{-1}+x_i\partial_{x_j}H^{-1}+x_ix_jH^{-1}+\delta_{ij}H^{-1},$$
which, at the level of kernels, means that
$$R_{ij}(x,z)=\partial^2_{x_i,x_j}F_{-1}(x,z)+x_j\partial_{x_i}F_{-1}(x,z)+x_i\partial_{x_j}F_{-1}(x,z)+x_ix_jF_{-1}(x,z)+\delta_{ij}F_{-1}(x,z),$$
By the estimates given in Lemmas \ref{Lem:F-sig est}, \ref{Lem:Rij est} and \ref{Lem:Bruno} of Section \ref{Section:Tecnica}, we can apply the statements of Section \ref{Section:Abstracta} to show that the operators $x_i\partial_{x_j}H^{-1}$, $x_ix_jH^{-1}$ and $H^{-1}$ are bounded on $C^{0,\alpha}_H$. Hence, $\partial^2_{x_i,x_j}H^{-1}$ maps $C^{0,\alpha}_H$ into $C^{0,\alpha}_H$ continuously. Observe now that the operator $\mathcal{R}_{i,-j}$, for $j\in\set{1,\ldots,n}$, can be written as
$$\mathcal{R}_{i,-j}=-\partial^2_{x_i,x_j}H^{-1}+x_j\partial_{x_i}H^{-1}-x_i\partial_{x_j}H^{-1}+x_ix_jH^{-1}+\delta_{ij}H^{-1}.$$
The observations above give the conclusion for $\mathcal{R}_{i,-j}$, $j\in\set{1,\ldots,n}$.
\end{proof}

For technical reasons we have to consider the first order adjoint Hermite-Riesz transforms, that are defined by
$$\mathcal{R}_i^\ast f(x)=H^{-1/2}A_if(x)=\int_{\Real^n}F_{-1/2}(x,z)(A_if)(z)~dz,\qquad f\in\mathcal{S},~x\in\Real^n,~1\leq\abs{i}\leq n.$$

\begin{thm}\label{Thm:Ad R1 Calfa}
The operators $\mathcal{R}_i^\ast$, $1\leq\abs{i}\leq n$, are bounded operators on $C^{0,\alpha}_H$, $0<\alpha<1$.
\end{thm}

\begin{proof}
Observe that, if $1\leq i\leq n$, then, by Remark \ref{Rem:H y Ai}, $R^\ast_{-i}=H^{-1/2}A_{-i}=A_{-i}(H+2)^{-1/2}$. This operator already appeared in the proof of Theorem \ref{Thm:R y Calfa}, and there we showed that it is a bounded operator on $C^{0,\alpha}_H$.

On the other hand, $R_{-i}^\ast=-H^{-1/2}\partial_{x_i}+H^{-1/2}x_i$. But by Lemmas \ref{Lem:F-sig est}, \ref{Lem:H-sig 1}, and \ref{Lem:Bruno} of Section \ref{Section:Tecnica}, and Proposition \ref{Prop:tecnico2}, we can see that the operator $f\mapsto H^{-1/2}x_if$, initially defined on $\mathcal{S}$, maps $C^{0,\alpha}_H$ into itself continuously. Therefore, we obtain the same conclusion for the operator $f\mapsto R_{-i}^\ast f-H^{-1/2}x_if=-H^{-1/2}\partial_{x_i}f$. Consequently, $R^\ast_i=H^{-1/2}A_i=H^{-1/2}\partial_{x_i}+H^{-1/2}x_i$ is a bounded operator on $C^{0,\alpha}_H$.
\end{proof}

\subsection{Proof of Theorem A}

We start with \textit{(A1)}. By recalling the results in Lemmas \ref{Lem:F sig est} and \ref{Lem:B sig est}, if we put $\gamma=2\sigma<1$ and $\varepsilon=\alpha-2\sigma$ in Proposition \ref{Prop:tecnico2}, we get the conclusion.

Consider now \textit{(A2)}. Using Remark \ref{Rem:H y Ai} and Theorem \ref{Thm:H-2k}, we have
$$H^\sigma u\in C_H^{1,\alpha-2\sigma}\Leftrightarrow A_iH^\sigma u,\,A_{-i}H^\sigma u\in C_H^{0,\alpha-2\sigma}\Leftrightarrow (H+2)^\sigma A_iu,\, (H-2)^\sigma A_{-i}u\in C_H^{0,\alpha-2\sigma}.$$
By Theorem \ref{Thm:H+2k}, together with Lemmas \ref{Lem:F sig est} and \ref{Lem:B sig est}, we can apply Proposition \ref{Prop:tecnico2}, with $\gamma=2\sigma<1$ and $\varepsilon=\alpha-2\sigma$, in order to get $(H+2)^\sigma:C_H^{0,\alpha}\to C_H^{0,\alpha-2\sigma}$ continuously, and then $\norm{(H+2)^\sigma A_iu}_{C_H^{0,\alpha-2\sigma}}\leq C\norm{A_iu}_{C^{0,\alpha}_H}\leq C\norm{u}_{C^{1,\alpha}_H}$. Applying Theorem \ref{Thm:H-2k}, Lemmas \ref{Lem:F sig est} and \ref{Lem:B sig est}, and Proposition \ref{Prop:tecnico2}, we get $\norm{(H-2)^\sigma A_{-i}u}_{C^{0,\alpha-2\sigma}_H}\leq C\norm{u}_{C^{1,\alpha}_H}$. Thus, $\norm{H^\sigma u}_{C^{1,\alpha-2\sigma}_H}\leq C\norm{u}_{C^{1,\alpha}_H}$.

Let us prove \textit{(A3)}. We can write
$$H^\sigma=H^{\sigma-1/2}\circ H^{-1/2}\circ H=H^{\sigma-1/2}\circ\frac{1}{2}\sum_{i=1}^n\left(\mathcal{R}_{-i}^\ast A_{-i}+\mathcal{R}_i^\ast A_i\right),$$
where $\mathcal{R}^\ast_{\pm i}$ are the adjoint Hermite-Riesz transforms, that are bounded operators on $C_H^{0,\alpha}$ (Theorem \ref{Thm:Ad R1 Calfa}). Consequently,
$$\frac{1}{2}\sum_{i=1}^n\left(\mathcal{R}_{-i}^\ast A_{-i}u+\mathcal{R}_i^\ast A_iu\right)=:v\in C_H^{0,\alpha}.$$
Now we distinguish two cases. If $\sigma-1/2>0$, then $0<\alpha-2(\sigma-1/2)<1$ by hypothesis, so we can apply \textit{(A1)} to obtain that $H^{\sigma-1/2}v\in C_H^{0,\alpha-2\sigma+1}$, and $\norm{H^\sigma u}_{C_H^{0,\alpha-2\sigma+1}}\leq C\norm{u}_{C_H^{1,\alpha}}$. If $\sigma-1/2<0$, then $0<\alpha+2(-\sigma+1/2)<1$, and we will get $H^{-(-\sigma+1/2)}v\in C^{0,\alpha-2\sigma+1}_H$, and $\norm{H^\sigma u}_{C^{0,\alpha-2\sigma+1}_H}\leq C\norm{u}_{C^{1,\alpha}_H}$, as soon as we have proved Theorem B, \textit{(B1)}. If $\sigma=1/2$, the result just follows from the boundedness of the adjoint Hermite-Riesz transforms on $C^{0,\alpha}_H$, Theorem \ref{Thm:Ad R1 Calfa}.

By iteration of \textit{(A1)}, \textit{(A2)} and \textit{(A3)}, and using Remark \ref{Rem:H y Ai} and Theorems \ref{Thm:H+2k} and \ref{Thm:H-2k}, we can derive \textit{(A4)}. The rather cumbersome details are left to the interested reader.

\subsection{Proof of Theorem B}

To prove \textit{(B1)} note that, if $\alpha+2\sigma\leq1$ then $0<\sigma<1/2$. Let us write
$$H^{-\sigma}u(x_1)-H^{-\sigma}u(x_2)=\int_{\Real^n}[u(z)-u(x_1)][F_{-\sigma}(x_1,z)-F_{-\sigma}(x_2,z)]dz
+u(x_1)\left[H^{-\sigma}1(x_1)-H^{-\sigma}1(x_2)\right].$$
By Lemma \ref{Lem:H-sig 1}, the second term above is bounded by $C[u]_{M^\alpha}\abs{x_1-x_2}^{\alpha+2\sigma}$. We split the remaining integral on $B=B(x_1,2\abs{x_1-x_2})$ and on $B^c$. We use Lemma \ref{Lem:F-sig est} to get
$$\int_B\abs{u(z)-u(x_1)}F_{-\sigma}(x_1,z)~dz\leq C[u]_{C^{0,\alpha}}\int_B\frac{\abs{x_1-z}^\alpha}{\abs{x_1-z}^{n-2\sigma}}~dz=C[u]_{C^{0,\alpha}}\abs{x_1-x_2}^{\alpha+2\sigma}.$$
Let $B'=B(x_2,4\abs{x_1-x_2})$. Then, by the triangle inequality,
\begin{align*}
  \int_B\abs{u(z)-u(x_1)}F_{-\sigma}(x_2,z)~dz &\leq C[u]_{C^{0,\alpha}}\int_{B'}\frac{\abs{x_1-z}^\alpha}{\abs{x_2-z}^{n-2\sigma}}~dz \\
   &\leq C[u]_{C^{0,\alpha}}\left[\abs{x_1-x_2}^\alpha\int_{B'}\frac{1}{\abs{x_2-z}^{n-2\sigma}}~dz+\int_{B'}\abs{x_2-z}^{\alpha-n+2\sigma}~dz\right] \\
   &= C[u]_{C^{0,\alpha}}\abs{x_1-x_2}^{\alpha+2\sigma}.
\end{align*}
Denote by $\widetilde{B}$ the ball with center $x_2$ and radius $\abs{x_1-x_2}$. Note that, for $z\in\widetilde{B}^c$, $\abs{z-x_1}<2\abs{z-x_2}$. Then, we apply Lemma \ref{Lem:F-sig est} to get
\begin{align*}
  \int_{B^c}\abs{u(z)-u(x_1)}\abs{F_{-\sigma}(x_1,z)-F_{-\sigma}(x_2,z)}dz &\leq C[u]_{C^{0,\alpha}}\abs{x_1-x_2}\int_{\widetilde{B}^c}\frac{\abs{z-x_2}^\alpha}{\abs{z-x_2}^{n+1-2\sigma}}~e^{-\frac{\abs{x-z}^2}{C}}dz \\
   &\leq C[u]_{C^{0,\alpha}}\abs{x_1-x_2}^{\alpha+2\sigma}.
\end{align*}
Thus, $[H^{-\sigma}u]_{C^{0,\alpha+2\sigma}}\leq C\norm{u}_{C^{0,\alpha}_H}.$ For the decay, we put
$$H^{-\sigma}u(x)=\int_{\overline{B}}(u(z)-u(x))F_{-\sigma}(x,z)~dz+\int_{\overline{B}^c}(u(z)-u(x))F_{-\sigma}(x,z)~dz+u(x)H^{-\sigma}1(x),$$
where $\overline{B}=B\left(x,\frac{1}{1+\abs{x}}\right)$. We have
$$\int_{\overline{B}}\abs{u(z)-u(x)}F_{-\sigma}(x,z)~dz\leq C[u]_{C^{0,\alpha}}\int_{\overline{B}}\frac{\abs{z-x}^{\alpha}}{\abs{x-z}^{n-2\sigma}}~dz\leq [u]_{C^{0,\alpha}}\frac{C}{(1+\abs{x})^{\alpha+2\sigma}},$$
and
$$\abs{\int_{\overline{B}^c}(u(z)-u(x))F_{-\sigma}(x,z)~dz+u(x)H^{-\sigma}1(x)}\leq\int_{\overline{B}^c}\abs{u(z)}F_{-\sigma}(x,z)~dz+2\abs{u(x)}\abs{H^{-\sigma}1(x)}.$$
To estimate the very last integral, we can proceed as we did for the integral part of the operator $T$ in the proof of Proposition \ref{Prop:tecnico2}, splitting the integral in annulus (details are left to the reader). By Lemma \ref{Lem:H-sig 1}, $\abs{u(x)H^{-\sigma}1(x)}\leq C[u]_{M^\alpha}(1+\abs{x})^{-(\alpha+2\sigma)}$. This concludes the proof of Theorem B \textit{(B1)}.

In order to prove \textit{(B2)}, we observe that, by using the boundedness of the first order Hermite-Riesz transforms on $C^{0,\alpha}_H$ (Theorem \ref{Thm:R y Calfa}), we get
$$\norm{A_iH^{-\sigma}u}_{C^{0,\alpha+2\sigma-1}_H}=\norm{\mathcal{R}_iH^{-\sigma+1/2}u}_{C^{0,\alpha+2\sigma-1}_H}\leq C\norm{H^{-\sigma+1/2}u}_{C^{0,\alpha+2\sigma-1}_H}\leq C\norm{u}_{C^{0,\alpha}_H},$$
where in the last inequality we applied Theorem A \textit{(A1)} if $-\sigma+1/2>0$, and the case \textit{(B1)} just proved above if $-\sigma+1/2<0$. The case $\sigma=1/2$ is contained in Theorem \ref{Thm:R y Calfa}.

Under the hypotheses of \textit{(B3)}, we have to prove that $A_iA_jH^{-\sigma}u$ belongs to $C^{0,\alpha+2\sigma-2}_H$. But $A_iA_jH^{-\sigma}u=\mathcal{R}_{ij}H^{1-\sigma}u$. Therefore, Theorem A \textit{(A1)}, and Theorem \ref{Thm:R y Calfa} give the result.

\section{Computational Lemmas}\label{Section:Tecnica}

\begin{lem}\label{Lem:exp est}
For each positive number $a$, let
$$\psi_{s,z}^a(x)=e^{-a\left[s\abs{x+z}^2+\frac{1}{s}\abs{x-z}^2\right]},\qquad x,z\in\Real^n,~s\in(0,1).$$
Then,
\begin{equation}\label{psi est}
\psi_{s,z}^a(x)\leq e^{-\frac{a}{4}\abs{x}\abs{x-z}}e^{-\frac{a}{4}\frac{\abs{x-z}^2}{s}}.
\end{equation}
\end{lem}

\begin{proof}
We have
$$\psi_{s,z}^a(x)\leq e^{-\frac{a}{2}\frac{\abs{x-z}^2}{s}}e^{-\frac{a}{2}\left[s\abs{x+z}^2+\frac{1}{s}\abs{x-z}^2\right]}\leq e^{-\frac{a}{2}\frac{\abs{x-z}^2}{s}}e^{-\frac{a}{2}\abs{x-z}\abs{x+z}}.$$
The first inequality above is obvious. For the second one, we argue as follows: if $\abs{x+z}\leq\abs{x-z}$ then it is clearly valid; when $\abs{x-z}<\abs{x+z}$ we minimize the function $\theta(s)=\frac{a}{2}\left[s\abs{x+z}^2+\frac{1}{s}\abs{x-z}^2\right]$, $s\in(0,1)$, to get $\theta(s)\geq\frac{a}{2}\abs{x-z}\abs{x+z}$. To obtain the desired estimate let us first assume that $x\cdot z>0$. Then $\abs{x+z}\geq\abs{x}$, and $e^{-\frac{a}{2}\frac{\abs{x-z}^2}{s}}e^{-\frac{a}{2}\abs{x-z}\abs{x+z}}\leq e^{-\frac{a}{4}\frac{\abs{x-z}^2}{s}}e^{-\frac{a}{4}\abs{x}\abs{x-z}}$. If $x\cdot z\leq0$, then $\abs{x-z}\geq\abs{x}$, and $e^{-\frac{a}{2}\frac{\abs{x-z}^2}{s}}e^{-\frac{a}{2}\abs{x-z}\abs{x+z}}\leq e^{-\frac{a}{4}\frac{\abs{x-z}^2}{s}}e^{-\frac{a}{4}\frac{\abs{x}\abs{x-z}}{s}}\leq e^{-\frac{a}{4}\frac{\abs{x-z}^2}{s}}e^{-\frac{a}{4}\abs{x}\abs{x-z}}$. Thus, \eqref{psi est} follows.
\end{proof}

\begin{rem}
Note that Lemma \ref{Lem:exp est} gives the estimate $\displaystyle G_{t(s)}(x,z)\leq C\left(\frac{1-s}{s}\right)^{n/2}e^{-\frac{\abs{x}\abs{x-z}}{C}}e^{-\frac{\abs{x-z}^2}{Cs}}$, $s\in(0,1)$, $x,z\in\Real^n$, which appeared in Lemma 5.10 of \cite{Stinga-Torrea}.
\end{rem}

\begin{lem}\label{Lem:B-T}
Let $\eta,\rho\in\Real$. Then, for all $x,z\in\Real^n$,
$$\int_0^1\left(\frac{1-s}{s}\right)^{n/2}\frac{1}{s^\eta}~e^{-C\left[s\abs{x+z}^2+\frac{1}{s}\abs{x-z}^2\right]}~d\mu_\rho(s)\leq Ce^{-\frac{\abs{x}\abs{x-z}}{C}}e^{-\frac{\abs{x-z}^2}{C}}\cdot I_{\eta,\rho}(x,z),$$
where
$$I_{\eta,\rho}(x,z)=
\left\{
  \begin{array}{ll}
    {\displaystyle\frac{1}{\abs{x-z}^{n+2\eta+2\rho}}}, & \hbox{if }~n/2+\eta+\rho>0, \\
    1+\log\left(\frac{C}{\abs{x-z}^2}\right)\chi_{\set{\frac{C}{\abs{x-z}^2}>1}}(x-z), & \hbox{if }~n/2+\eta+\rho=0, \\
    1, & \hbox{if }~n/2+\eta+\rho<0. \\
  \end{array}
\right.
$$
\end{lem}

\begin{proof}
By \eqref{psi est}, we get
\begin{align*}
\lefteqn{\int_0^1\left(\frac{1-s}{s}\right)^{n/2}\frac{1}{s^\eta}~e^{-C\left[s\abs{x+z}^2+\frac{1}{s}\abs{x-z}^2\right]}~d\mu_\rho(s)}\\ &\leq Ce^{-\frac{\abs{x}\abs{x-z}}{C}}\left(\int_0^{1/2}\frac{1}{s^{n/2+\eta+\rho}}~e^{-\frac{\abs{x-z}^2}{Cs}}~\frac{ds}{s}+ e^{-\frac{\abs{x-z}^2}{C}}\int_{1/2}^1(1-s)^{n/2}~d\mu_\rho(s)\right)  \\
   &= Ce^{-\frac{\abs{x}\abs{x-z}}{C}}\left(\frac{C}{\abs{x-z}^{n+2\eta+2\rho}}\int_{\frac{\abs{x-z}^2}{C}}^\infty r^{n/2+\eta+\rho}e^{-2r}~\frac{dr}{r}+ Ce^{-\frac{\abs{x-z}^2}{C}}\right) \\
   &\leq Ce^{-\frac{\abs{x}\abs{x-z}}{C}}\left(\frac{e^{-\frac{\abs{x-z}^2}{C}}}{\abs{x-z}^{n+2\eta+2\rho}}\int_{\frac{\abs{x-z}^2}{C}}^\infty r^{n/2+\eta+\rho}e^{-r}~\frac{dr}{r}+e^{-\frac{\abs{x-z}^2}{C}}\right) \\
   &= Ce^{-\frac{\abs{x}\abs{x-z}}{C}}\cdot I.
\end{align*}

If $n/2+\eta+\rho>0$, we have
$$I\leq\frac{e^{-\frac{\abs{x-z}^2}{C}}}{\abs{x-z}^{n+2\eta+2\rho}}\int_0^\infty r^{n/2+\eta+\rho}e^{-r}~\frac{dr}{r}+e^{-\frac{\abs{x-z}^2}{C}}\leq \frac{C}{\abs{x-z}^{n+2\eta+2\rho}}~e^{-\frac{\abs{x-z}^2}{C}}.$$

Suppose now that $n/2+\eta+\rho\leq0$. Consider two cases: if $\frac{\abs{x-z}^2}{C}\geq1$, then,
$$I\leq e^{-\frac{\abs{x-z}^2}{C}}\left[\frac{1}{\abs{x-z}^{n+2\eta+2\rho}}\int_1^\infty r^{n/2+\eta+\rho}e^{-r}\frac{dr}{r}+1\right]=Ce^{-\frac{\abs{x-z}^2}{C}}\left[\frac{C}{\abs{x-z}^{n+2\eta+2\rho}}+1\right]\leq C,$$
and, when $\frac{\abs{x-z}^2}{C}<1$, we have
\begin{align*}
I &\leq e^{-\frac{\abs{x-z}^2}{C}}\left[\frac{1}{\abs{x-z}^{n+2\eta+2\rho}}\left(\int_{\frac{\abs{x-z}^2}{C}}^1r^{n/2+\eta+\rho}~\frac{dr}{r}+C\right)+1\right] \\
&\leq Ce^{-\frac{\abs{x-z}^2}{C}}\cdot
\left\{
  \begin{array}{ll}
    1+\log\left(\frac{C}{\abs{x-z}^2}\right), & \hbox{if }~n/2+\eta+\rho=0, \\
    1, & \hbox{if }~n/2+\eta+\rho<0.
  \end{array}
\right.
\end{align*}
\end{proof}

\begin{lem}\label{Lem:F sig est}
Denote by $\mathcal{F}$ any of the kernels $F_\sigma(x,z)$ (defined in \eqref{FyB}), or $F_{\pm 2k,\sigma}(x,z)$ (given in Theorems \ref{Thm:H+2k} and \ref{Thm:H-2k}). Then,
\begin{equation}\label{F sig est}
\abs{\mathcal{F}(x,z)}\leq\frac{C}{\abs{x-z}^{n+2\sigma}}~e^{-\frac{\abs{x}\abs{x-z}}{C}}e^{-\frac{\abs{x-z}^2}{C}},
\end{equation}
for all $x,z\in\Real^n$, and
\begin{equation}\label{F sig suave}
\abs{\mathcal{F}(x_1,z)-\mathcal{F}(x_2,z)}\leq \frac{C\abs{x_1-x_2}}{\abs{x_2-z}^{n+1+2\sigma}}~e^{-\frac{\abs{z}\abs{x_2-z}}{C}}e^{-\frac{\abs{x_2-z}^2}{C}},
\end{equation}
for all $x_1,x_2\in\Real^n$ such that $\abs{x_1-z}>2\abs{x_1-x_2}$.
\end{lem}

\begin{proof}
Let us first consider $\mathcal{F}=F_\sigma$. The estimate in \eqref{F sig est} is already stated in \cite{Stinga-Torrea}, Lemma 5.11. Nevertheless, we can prove it here quickly by using, in \eqref{F sig s}, Lemma \ref{Lem:B-T}, with $\eta=0$ and $\rho=\sigma$. To get \eqref{F sig suave}, we observe that, by the Mean Value Theorem,
\begin{equation}\label{G_s suave}
\abs{G_{t(s)}(x_1,z)-G_{t(s)}(x_2,z)}\leq C\abs{x_1-x_2}\left(\frac{1-s}{s}\right)^{n/2}\frac{1}{s^{1/2}}~e^{-\frac{1}{8}\left[s\abs{\xi+z}^2+\frac{1}{s}\abs{\xi-z}^2\right]},
\end{equation}
for some $\xi=(1-\lambda)x_1+\lambda x_2$, $\lambda\in[0,1]$. Then, by Lemma \ref{Lem:B-T}, with $\eta=1/2$ and $\rho=\sigma$,
\begin{align*}
  \abs{F_\sigma(x_1,z)-F_\sigma(x_2,z)} &\leq \abs{x_1-x_2}\sup_{\set{\xi=(1-\lambda)x_1+\lambda x_2:\lambda\in[0,1]}}\abs{\nabla_xF_\sigma(\xi,z)} \\
   &\leq C\abs{x_1-x_2}\sup_\xi\int_0^1\left(\frac{1-s}{s}\right)^{n/2}\frac{1}{s^{1/2}}~e^{-\frac{1}{8}\left[s\abs{\xi+z}^2+\frac{1}{s}\abs{\xi-z}^2\right]}~d\mu_\sigma(s) \\
   &\leq C\abs{x_1-x_2}\sup_\xi\frac{1}{\abs{\xi-z}^{n+1+2\sigma}}~e^{-\frac{\abs{z}\abs{\xi-z}}{C}}e^{-\frac{\abs{\xi-z}^2}{C}} \\
   &\leq C\frac{\abs{x_1-x_2}}{\abs{x_2-z}^{n+1+2\sigma}}~e^{-\frac{\abs{z}\abs{x_2-z}}{C}}e^{-\frac{\abs{x_2-z}^2}{C}},
\end{align*}
where in the last inequality we used that $\abs{\xi-z}\geq\frac{1}{2}\abs{x_2-z}$, since $\abs{x_1-z}>2\abs{x_1-x_2}$. In a similar way we can prove both estimates for $\mathcal{F}=F_{2k,\sigma}$, because $0\leq F_{2k,\sigma}(x,z)\leq F_\sigma(x,z)$, and the details are left to the reader. Note that, by \eqref{dif1}-\eqref{dif2}, up to a multiplicative constant we have
\begin{align*}
    \abs{F_{-2k,\sigma}(x,z)} &= \abs{\int_0^{1/2}\left(\frac{1+s}{1-s}\right)^kG_{t(s)}(x,z)~d\mu_\sigma(s)+\int_{1/2}^1\left(\frac{1+s}{1-s}\right)^k\left[G_{t(s)}(x,z)-\phi_{2k}(x)\right]d\mu_\sigma(s)} \\
    &\leq C\left[F_\sigma(x,z)+e^{-\frac{\abs{x}\abs{x-z}}{C}}e^{-\frac{\abs{x-z}^2}{C}}\int_{1/2}^1\left(\frac{1-s}{1+s}\right)^{n/2}d\mu_\sigma(s)\right],
\end{align*}
and therefore, \eqref{F sig est} is valid for $F_{-2k,\sigma}$. By \eqref{G_s suave},
\begin{align}
\nonumber\lefteqn{\int_0^{1/2}\left(\frac{1+s}{1-s}\right)^k\abs{G_{t(s)}(x_1,z)-G_{t(s)}(x_2,z)}d\mu_\sigma(s)} \\ &\leq \label{0-1/2suave} C\abs{x_1-x_2}\sup_\xi\int_0^1\left(\frac{1-s}{s}\right)^{n/2}\frac{1}{s^{1/2}}~e^{-\frac{1}{8}\left[s\abs{\xi+z}^2+\frac{1}{s}\abs{\xi-z}^2\right]}~d\mu_\sigma(s).
\end{align}
Recall the definition of $M_r(x,z)$ given in \eqref{M_r}. It can be checked that
$$\abs{\frac{d^k}{dr^k}\nabla_x M_r(x,z)}\leq Ce^{-\frac{\abs{x}\abs{x-z}}{C}}e^{-\frac{\abs{x-z}^2}{C}},\qquad r\in(0,1/3).$$
Thus, by Taylor's formula,
\begin{equation}\label{nabla dif}
\abs{\nabla_x\left[G_{t(s)}(x,z)-\phi_{2k}(x,z,s)\right]}\leq C\left(\frac{1-s}{1+s}\right)^{k+n/2}e^{-\frac{\abs{x}\abs{x-z}}{C}}e^{-\frac{\abs{x-z}^2}{C}},\qquad s\in(1/2,1),
\end{equation}
and, consequently, when $\abs{x_1-z}>2\abs{x_1-x_2}$,
\begin{align}
    \nonumber\lefteqn{\int_{1/2}^1\left(\frac{1+s}{1-s}\right)^k\abs{(G_{t(s)}(x_1,z)-\phi_{2k}(x_1,z,s))-(G_{t(s)}(x_2,z)-\phi_{2k}(x_2,z,s))}d\mu_\sigma(s)} \\
    &\leq
\nonumber C\abs{x_1-x_2}\sup_\xi\int_{1/2}^1\left(\frac{1+s}{1-s}\right)^k\abs{\nabla_x\left[G_{t(s)}(\xi,z)-\phi_{2k}(\xi,z,s)\right]}~d\mu_\sigma(s) \\
   \label{1/2-1suave} &\leq
C\abs{x_1-x_2}\sup_\xi e^{-\frac{\abs{z}\abs{\xi-z}}{C}}e^{-\frac{\abs{\xi-z}^2}{C}}\leq  C\abs{x_1-x_2}e^{-\frac{\abs{z}\abs{x_2-z}}{C}}e^{-\frac{\abs{x_2-z}^2}{C}}.
\end{align}
Pasting estimates \eqref{0-1/2suave} and \eqref{1/2-1suave}, \eqref{F sig suave} follows for $\mathcal{F}=F_{-2k,\sigma}$.
\end{proof}

\begin{lem}\label{Lem:B sig est}
Denote by $\mathcal{B}$ any of the functions $B_\sigma$ or $B_{\pm 2k,\sigma}$ defined in \eqref{FyB} and in Theorems \ref{Thm:H+2k} and \ref{Thm:H-2k}. Then $\mathcal{B}\in C^\infty(\Real^n)$ and, for all $x\in\Real^n$,
\begin{equation}\label{B sig est+suav}
\abs{\mathcal{B}(x)}\leq C\left(1+\abs{x}^{2\sigma}\right),\qquad\hbox{ and }\qquad\abs{\nabla\mathcal{B}(x)}\leq C
\left\{
  \begin{array}{ll}
    \abs{x}, & \hbox{if }\abs{x}\leq1, \\
    \abs{x}^{2\sigma-1}, & \hbox{if }\abs{x}>1.
  \end{array}
\right.
\end{equation}
\end{lem}

\begin{proof}
The first inequality in \eqref{B sig est+suav} for the case $\mathcal{B}=B_\sigma$ is contained in \cite{Stinga-Torrea}, Lemma 5.11. The identity
$$e^{-tH}1(x)=\int_{\Real^n}G_t(x,z)~dz=\frac{1}{(2\pi\cosh2t)^{n/2}}~e^{-\frac{\tanh2t}{2}\abs{x}^2},$$
(stated in \cite{Harboure-deRosa-Segovia-Torrea}) and Meda's change of parameters \eqref{Meda transf} give
\begin{eqnarray*}
  B_\sigma(x) = \frac{1}{\Gamma(-\sigma)}\int_0^1\left[\left(\frac{1-s^2}{2\pi(1+s^2)}\right)^{n/2}e^{-\frac{s}{1+s^2}\abs{x}^2}-1\right]~d\mu_\sigma(s).
\end{eqnarray*}
We differentiate under the integral sign to see that $B_\sigma\in C^\infty(\Real^n)$, and
$$\abs{\nabla B_\sigma(x)}=\abs{\frac{2x}{\Gamma(-\sigma)}\int_0^1\frac{s}{1+s^2}\left(\frac{1-s^2}{2\pi(1+s^2)}\right)^{n/2}e^{-\frac{s}{1+s^2}\abs{x}^2}~d\mu_\sigma(s)} \leq C\abs{x}\int_0^1se^{-\frac{s}{2}\abs{x}^2}d\mu_\sigma(s)=:\widetilde{I}(x).$$
By \eqref{est dmu rho},
$$\widetilde{I}(x)\leq C\abs{x}\left[\int_0^{1/2}e^{-\frac{s}{2}\abs{x}^2}~\frac{ds}{s^\sigma}+ e^{-\frac{\abs{x}^2}{C}}\int_{1/2}^1~d\mu_\sigma(s)\right]= C\abs{x}^{2\sigma-1}\int_0^{\frac{\abs{x}^2}{4}}e^{-r}~\frac{dr}{r^\sigma}+C\abs{x}e^{-\frac{\abs{x}^2}{C}}.$$
If $\abs{x}\leq1$,
$$\int_0^{\frac{\abs{x}^2}{4}}e^{-r}~\frac{dr}{r^\sigma}\leq\int_0^{\frac{\abs{x}^2}{4}}~\frac{dr}{r^\sigma}=C\abs{x}^{2-2\sigma},$$
and, if $\abs{x}>1$,
$$\int_0^{\frac{\abs{x}^2}{4}}e^{-r}~\frac{dr}{r^\sigma}\leq\int_0^{1/4}~\frac{dr}{r^\sigma}+ \int_{1/4}^{\frac{\abs{x}^2}{4}}e^{-r}~dr=C-e^{-\frac{\abs{x}^2}{C}}\leq C.$$
Hence, \eqref{B sig est+suav} with $\mathcal{B}=B_\sigma$ is proved.

We can write
\begin{align*}
    B_{2k,\sigma}(x) &= \frac{1}{\Gamma(-\sigma)}\int_0^1\left[\left(\frac{1-s}{1+s}\right)^k\left(\frac{1-s^2}{2\pi(1+s^2)}\right)^{n/2}-1\right]e^{-\frac{s}{1+s^2}\abs{x}^2}~d\mu_\sigma(s) \\
    & \quad+~\frac{1}{\Gamma(-\sigma)}\int_0^1\left(e^{-\frac{s}{1+s^2}\abs{x}^2}-1\right)~d\mu_\sigma(s)=:I+II.
\end{align*}
The bounds for $I$ and $II$ can be deduced as in the proof of Lemma 5.11 of \cite{Stinga-Torrea}. We give the calculation for completeness. For both terms we use \eqref{est dmu rho} and the Mean Value Theorem. That is,
$$\abs{I}\leq C\int_0^{1/2}\abs{\left(\frac{1-s}{1+s}\right)^k\left(\frac{1-s^2}{2\pi(1+s^2)}\right)^{n/2}-1}\frac{ds}{s^{\sigma+1}}+\int_{1/2}^1~d\mu_\sigma(s)\leq C\int_0^{1/2}s~\frac{ds}{s^{1+\sigma}}+C=C.$$
For $II$ we have to consider two cases. Assume first that $\abs{x}^2\leq2$. Then,
$$\abs{II}\leq C\int_0^{1/2}\abs{e^{-\frac{s}{1+s^2}\abs{x}^2}-1}~\frac{ds}{s^{1+\sigma}}+\int_{1/2}^1~d\mu_\sigma(s)\leq C\int_0^{1/2}\abs{x}^2s~\frac{ds}{s^{1+\sigma}}+C\leq C.$$
In the case $\abs{x}^2>2$,
\begin{align*}
  \abs{II} &\leq \abs{x}^2\int_0^{\frac{1}{\abs{x}^2}}s~\frac{ds}{s^{1+\sigma}}+\int_{\frac{1}{\abs{x}^2}}^1~d\mu_\sigma(s)\leq \abs{x}^2\int_0^{\frac{1}{\abs{x}^2}}s^{-\sigma}~ds+\int_\frac{1}{\abs{x}^2}^1\frac{ds}{(1-s)\left(-\log(1-s)\right)^{1+\sigma}} \\
   &= C\abs{x}^{2\sigma}+C\left[-\log\left(1-\frac{1}{\abs{x}^2}\right)\right]^{-\sigma}\leq C\abs{x}^{2\sigma},
\end{align*}
since $-\log(1-s)\sim s$, as $s\sim0$. On the other hand, observe that $\abs{\nabla B_{2k,\sigma}(x)}\leq\widetilde{I}(x)$. Thus, \eqref{B sig est+suav} follows with $\mathcal{B}=B_{2k,\sigma}$.

When $\mathcal{B}=B_{-2k,\sigma}$,
\begin{align*}
    B_{-2k,\sigma}(x) &= \frac{1}{\Gamma(-\sigma)}\int_0^{1/2}\left[\left(\frac{1+s}{1-s}\right)^k\left(\frac{1-s^2}{2\pi(1+s^2)}\right)^{n/2}e^{-\frac{s}{1+s^2}\abs{x}^2}-1\right]~d\mu_\sigma(s) \\
    &\quad +~\frac{1}{\Gamma(-\sigma)}\int_{1/2}^1\left[\left(\frac{1+s}{1-s}\right)^k\int_{\Real^n}\left(G_{t(s)}(x,z)-\phi_{2k}(x,z,s)\right)~dz-1\right]~d\mu_\sigma(s) \\
    &= III+IV.
\end{align*}
If we write $III$ as
$$\frac{1}{\Gamma(-\sigma)}\int_0^{1/2}\left[\left(\frac{1+s}{1-s}\right)^k\left(\frac{1-s^2}{2\pi(1+s^2)}\right)^{n/2}-1\right]~d\mu_\sigma(s)+\frac{1}{\Gamma(-\sigma)}\int_0^{1/2}\left(e^{-\frac{s}{1+s^2}\abs{x}^2}-1\right)~d\mu_\sigma(s),$$
then, we can handle these two terms as we did for $I$ and $II$ above to get $\abs{III}\leq C\left(1+\abs{x}^{2\sigma}\right)$. By \eqref{dif1}-\eqref{dif2}, $\abs{IV}\leq C$. For the gradient of $B_{-2k,\sigma}$, similar estimates to those used for $\nabla B_\sigma$ can be applied for the term $\nabla_xIII$. Finally, \eqref{nabla dif} implies that $\abs{\nabla_xIV}\leq C$. The proof is complete.
\end{proof}

The following Lemma contains a small refinement of the estimate for the kernel $F_{-\sigma}(x,z)$ given in \cite{Bongioanni-Torrea}, Proposition 2.

\begin{lem}\label{Lem:F-sig est}
Take $\sigma\in(0,1]$. Then, for all $x,z\in\Real^n$,
\begin{equation}\label{F-sig est}
0\leq F_{-\sigma}(x,z)\leq C
\left\{
  \begin{array}{ll}
    \frac{1}{\abs{x-z}^{n-2\sigma}}~e^{-\frac{\abs{x}\abs{x-z}}{C}}e^{-\frac{\abs{x-z}^2}{C}}, & \hbox{ if }n>2\sigma, \\
    e^{-\frac{\abs{x}\abs{x-z}}{C}}e^{-\frac{\abs{x-z}^2}{C}}\left[1+\log\left(\frac{C}{\abs{x-z}^2}\right)\chi_{\set{\frac{C}{\abs{x-z}^2}>1}}(x-z)\right], & \hbox{ if }n=2\sigma, \\
    e^{-\frac{\abs{x}\abs{x-z}}{C}}e^{-\frac{\abs{x-z}^2}{C}}, & \hbox{ if }n<2\sigma.
  \end{array}
\right.
\end{equation}
If $\mathrm{F}(x,z)$ denotes any of the kernels $\nabla_xF_{-\sigma}(x,z)$, $x_iF_{-\sigma}(x,z)$ or $z_iF_{-\sigma}(x,z)$, then,
\begin{equation}\label{Ai F-sig est}
\abs{\mathrm{F}(x,z)}\leq C
\left\{
  \begin{array}{ll}
    \frac{1}{\abs{x-z}^{n+1-2\sigma}}~e^{-\frac{\abs{x}\abs{x-z}}{C}}e^{-\frac{\abs{x-z}^2}{C}}, & \hbox{ if }n>2\sigma-1, \\
    e^{-\frac{\abs{x}\abs{x-z}}{C}}e^{-\frac{\abs{x-z}^2}{C}}\left[1+\log\left(\frac{C}{\abs{x-z}^2}\right)\chi_{\set{\frac{C}{\abs{x-z}^2}>1}}(x-z)\right], & \hbox{ if }n=2\sigma-1.
  \end{array}
\right.
\end{equation}
Moreover, when $\abs{x_1-z}\geq2\abs{x_1-x_2}$,
\begin{multline}\label{F-sig suave}
\abs{F_{-\sigma}(x_1,z)-F_{-\sigma}(x_2,z)} \\ \leq C\abs{x_1-x_2}
\left\{
  \begin{array}{ll}
    \frac{1}{\abs{x_2-z}^{n+1-2\sigma}}~e^{-\frac{\abs{z}\abs{x_2-z}}{C}}e^{-\frac{\abs{x_2-z}^2}{C}}, & \hbox{if }\sigma\neq1, \\
    e^{-\frac{\abs{z}\abs{x_2-z}}{C}}e^{-\frac{\abs{x_2-z}^2}{C}}\left[1+\log\left(\frac{C}{\abs{x-z}^2}\right)\chi_{\set{\frac{C}{\abs{x-z}^2}>1}}(x-z)\right], & \hbox{if }\sigma=1,
  \end{array}
\right.
\end{multline}
and,
\begin{equation}\label{nab F-sig suav}
\abs{\mathrm{F}(x_1,z)-\mathrm{F}(x_2,z)}\leq C\frac{\abs{x_1-x_2}}{\abs{x_2-z}^{n+2-2\sigma}}~e^{-\frac{\abs{z}\abs{x_2-z}}{C}}e^{-\frac{\abs{x_2-z}^2}{C}}.
\end{equation}
\end{lem}

\begin{proof}
By \eqref{Meda transf},
\begin{equation}\label{F-sig}
F_{-\sigma}(x,z)=\frac{1}{\Gamma(\sigma)}\int_0^1G_{t(s)}(x,z)~d\mu_{-\sigma}(s)=C\int_0^1\left(\frac{1-s^2}{s}\right)^{n/2}e^{-\frac{1}{4}\left[s\abs{x+z}^2+\frac{1}{s}\abs{x-z}^2\right]}~d\mu_{-\sigma}(s).
\end{equation}
Then apply Lemma \ref{Lem:B-T}, with $\eta=0$ and $\rho=-\sigma$, to get \eqref{F-sig est}. Differentiation with respect to $x$ inside the integral in \eqref{F-sig} gives
\begin{equation}\label{nab F-sig}
\abs{\nabla_xF_{-\sigma}(x,z)}\leq C\int_0^1\left(\frac{1-s}{s}\right)^{n/2}\frac{1}{s^{1/2}}~e^{-\frac{1}{8}\left[s\abs{x+z}^2+\frac{1}{s}\abs{x-z}^2\right]}~d\mu_{-\sigma}(s),
\end{equation}
and then Lemma \ref{Lem:B-T}, with $\eta=1/2$ and $\rho=-\sigma$, implies \eqref{Ai F-sig est} with $\mathrm{F}(x,z)=\nabla_xF_{-\sigma}(x,z)$. Take $x,z\in\Real^n$. If $x\cdot z\geq 0$, then $\abs{x}\leq\abs{x+z}$ and, in this situation, $\abs{x}F_{-\sigma}(x,z)$ is bounded by the RHS of \eqref{nab F-sig}. If $x\cdot z<0$, we have $\abs{x}\leq\abs{x-z}$, and in this case
$$\abs{x}F_{-\sigma}(x,z)\leq C\abs{x-z}e^{-\frac{1}{8}\abs{x-z}^2}\int_0^1\left(\frac{1-s}{s}\right)^{n/2}e^{-\frac{1}{8}\left[s\abs{x+z}^2+\frac{1}{s}\abs{x-z}^2\right]}~d\mu_{-\sigma}(s).$$
Therefore, by Lemma \ref{Lem:B-T}, we obtain \eqref{Ai F-sig est} for $\mathrm{F}(x,z)=x_iF_{-\sigma}(x,z)$. The same reasoning applies to $\mathrm{F}(x,z)=z_iF_{-\sigma}(x,z)$, since $\abs{z}\leq\abs{z-x}+\abs{x}$. To derive \eqref{F-sig suave}, we follow the proof of \eqref{F sig suave} in Lemma \ref{Lem:F sig est}, with $-\sigma$ in the place of $\sigma$, and we use Lemma \ref{Lem:B-T}. Estimate \eqref{nab F-sig suav} for $\mathrm{F}(x,z)=\nabla_xF_{-\sigma}(x,z)$ can be deduced by using the Mean Value Theorem and Lemma \ref{Lem:B-T}, since
\begin{multline*}
    \partial^2_{x_i,x_j}F_{-\sigma}(x,z)=\frac{1}{\Gamma(\sigma)}\int_0^1\left(\frac{1-s^2}{4\pi s}\right)^{n/2}e^{-\frac{1}{4}\left[s\abs{x+z}^2+\frac{1}{s}\abs{x-z}^2\right]}\times \\
\times\left[\left(-\frac{s}{2}(x_i+z_i)-\frac{1}{2s}(x_i-z_i)\right)\left(-\frac{s}{2}(x_j+z_j)-\frac{1}{2s}(x_j-z_j)\right)+\delta_{ij}\left(-\frac{s}{2}-\frac{1}{2s}\right)\right]~d\mu_{-\sigma}(s),
\end{multline*}
gives that
\begin{equation}\label{D2 H-sig}
\abs{D^2_xF_{-\sigma}(x,z)}\leq C\int_0^1\left(\frac{1-s}{s}\right)^{n/2}\frac{1}{s}~e^{-C\left[s\abs{x+z}^2+\frac{1}{s}\abs{x-z}^2\right]}~d\mu_{-\sigma}(s).
\end{equation}
Similar ideas can also be used to prove \eqref{nab F-sig suav} when $\mathrm{F}(x,z)$ is either $x_iF_{-\sigma}(x,z)$ or $z_iF_{-\sigma}(x,z)$. We skip the details.
\end{proof}

\begin{lem}\label{Lem:H-sig 1}
The function $H^{-\sigma}1$ belongs to the space $C^\infty(\Real^n)$, and
$$\abs{H^{-\sigma}1(x)}\leq\frac{C}{(1+\abs{x})^{2\sigma}},\qquad\hbox{and}\qquad\abs{\nabla H^{-\sigma}1(x)}\leq\frac{C}{(1+\abs{x})^{1+2\sigma}}.$$
\end{lem}

\begin{proof}
Observe that \eqref{Meda transf} applied to \eqref{H-sig1} gives
\begin{equation}\label{H-sig1 s}
H^{-\sigma}1(x)=\frac{1}{\Gamma(\sigma)}\int_0^1e^{-t(s)H}1(x)~d\mu_{-\sigma}(s)=C\int_0^1\left(\frac{1-s^2}{1+s^2}\right)^{n/2}e^{-\frac{s}{1+s^2}\abs{x}^2}~d\mu_{-\sigma}(s).
\end{equation}
Since
\begin{equation}\label{nab heat1}
\abs{\nabla_x e^{-t(s)H}1(x)}=2\abs{x}\frac{s}{1+s^2}\left(\frac{1-s^2}{2\pi(1+s^2)}\right)^{n/2}e^{-\frac{s}{1+s^2}\abs{x}^2}\leq Cs^{1/2}e^{-\frac{s}{C}\abs{x}^2},
\end{equation}
differentiation inside the integral sign in \eqref{H-sig1 s} is justified. By repeating this argument we obtain $H^{-\sigma}1\in C^\infty(\Real^n)$. To study the size of  $H^{-\sigma}1$, note that we can restrict to the case $\abs{x}>1$, because $H^{-\sigma}1$ is a continuous function. By \eqref{est dmu rho}, we have
$$\int_0^{1/2}\left(\frac{1-s^2}{1+s^2}\right)^{n/2}e^{-\frac{s}{1+s^2}\abs{x}^2}d\mu_{-\sigma}(s)\leq C\int_0^{1/2}e^{-\frac{s}{C}\abs{x}^2}\frac{ds}{s^{1-\sigma}}= C\abs{x}^{-2\sigma}\int_0^{\frac{\abs{x}^2}{2C}}e^{-r}\frac{dr}{r^{1-\sigma}}\leq C\abs{x}^{-2\sigma},$$
and
$$\int_{1/2}^1\left(\frac{1-s^2}{1+s^2}\right)^{n/2}e^{-\frac{s}{1+s^2}\abs{x}^2}d\mu_{-\sigma}(s)\leq Ce^{-C\abs{x}^2}\int_{1/2}^1\frac{(1-s)^{n/2-1}}{(-\log(1-s))^{1-\sigma}}~ds=Ce^{-C\abs{x}^2}.$$
Plugging these two estimates into \eqref{H-sig1 s} we get the bound for $H^{-\sigma}1$. For the growth of the gradient, we can use \eqref{nab heat1} and similar estimates as above to obtain the result.
\end{proof}

\begin{lem}\label{Lem:Rij est}
For $1\leq\abs{i},\abs{j}\leq n$, denote by $\mathrm{R}(x,z)$ any of the kernels $\partial^2_{x_i,x_j}F_{-1}(x,z)$, $x_i\partial_{x_j}F_{-1}(x,z)$ or $x_ix_jF_{-1}(x,z)$. Then
\begin{equation}\label{R size}
\abs{\mathrm{R}(x,z)}\leq\frac{C}{\abs{x-z}^n}~e^{-\frac{\abs{x}\abs{x-z}}{C}}e^{-\frac{\abs{x-z}^2}{C}},
\end{equation}
and, when $\abs{x_1-z}\geq2\abs{x_1-x_2}$,
\begin{equation}\label{Rij suave}
\abs{\mathrm{R}(x_1,z)-\mathrm{R}(x_2,z)}\leq C\frac{\abs{x_1-x_2}}{\abs{x_2-z}^{n+1}}~e^{-\frac{\abs{z}\abs{x_2-z}}{C}}e^{-\frac{\abs{x_2-z}^2}{C}}.
\end{equation}
As a consequence, the kernel of the second order Hermite-Riesz transforms $R_{ij}(x,z)=A_iA_jF_{-1}(x,z)$, also satisfies these size and smoothness estimates.
\end{lem}

\begin{proof}
We put $\sigma=1$ in \eqref{D2 H-sig} and we use Lemma \ref{Lem:B-T}, with $\eta=1$ and $\rho=-1$, to obtain the desired estimate for $D^2_xF_{-1}$. From \eqref{nab F-sig},
\begin{equation}\label{x nab F-1}
\abs{x_i\partial_{x_j}F_{-1}(x,z)}\leq C\abs{x}\int_0^1\left(\frac{1-s}{s}\right)^{n/2}\frac{1}{s^{1/2}}~e^{-\frac{1}{8}\left[s\abs{x+z}^2+\frac{1}{s}\abs{x-z}^2\right]}~d\mu_{-1}(s),
\end{equation}
If $\abs{x}\leq2$, then Lemma \ref{Lem:B-T}, with $\eta=1/2$ and $\rho=-1$, applied to \eqref{x nab F-1} gives
$$\abs{x_i\partial_{x_j}F_{-1}(x,z)}\leq\frac{C}{\abs{x-z}^{n-1}}~e^{-\frac{\abs{x}\abs{x-z}}{C}}e^{-\frac{\abs{x-z}^2}{C}}.$$
Assume that $\abs{x}>2$ in \eqref{x nab F-1}. Consider first the case $\abs{x}<2\abs{x-z}$, then by Lemma \ref{Lem:B-T},
$$\abs{x_i\partial_{x_j}F_{-1}(x,z)}\leq C\int_0^1\left(\frac{1-s}{s}\right)^{n/2}e^{-C\left[s\abs{x+z}^2+\frac{1}{s}\abs{x-z}^2\right]}~d\mu_{-1}(s)\leq \frac{C}{\abs{x-z}^{n-2}}~e^{-\frac{\abs{x}\abs{x-z}}{C}}e^{-\frac{\abs{x-z}^2}{C}}.$$
In the other case, namely $\abs{x}\geq2\abs{x-z}$, we use the fact that $\abs{x}>2$ to see that $\abs{x+z}^2=2\abs{x}^2-\abs{x-z}^2+2\abs{z}^2>\abs{x}^2$. Hence,
$$\abs{x_j\partial_{x_j}F_{-1}(x,z)}\leq C\int_0^1\left(\frac{1-s}{s}\right)^{n/2}\frac{1}{s^{1/2}}~e^{-C\left[s\abs{x+z}^2+\frac{1}{s}\abs{x-z}^2\right]}~d\mu_{-1}(s)\leq \frac{C}{\abs{x-z}^n}~e^{-\frac{\abs{x}\abs{x-z}}{C}}e^{-\frac{\abs{x-z}^2}{C}}.$$
Collecting terms, we have \eqref{R size} for $\mathrm{R}(x,z)=x_j\partial_{x_j}F_{-1}(x,z)$. Finally, to obtain \eqref{R size} with $\mathrm{R}(x,z)=x_ix_jF_{-1}(x,z)$, we note that by \eqref{F-sig},
$$\abs{x_ix_jF_{-1}(x,z)}\leq C\abs{x}^2\int_0^1\left(\frac{1-s}{s}\right)^{n/2}e^{-\frac{1}{8}\left[s\abs{x+z}^2+\frac{1}{s}\abs{x-z}^2\right]}~d\mu_{-1}(s),$$
and we consider the cases $\abs{x}\leq2$ and $\abs{x}>2$ as before. In the second situation, we assume first that $\abs{x}\leq2\abs{x-z}$ and, then, that $\abs{x}\geq2\abs{x-z}$ (which implies $\abs{x}\leq\abs{x+z}$), and we use the method of the proof given for $x_j\partial_{x_i}F_{-1}$ above.

To prove \eqref{Rij suave} we can use the Mean Value Theorem and Lemma \ref{Lem:exp est} (see the proof of \eqref{F sig suave} and \eqref{nab F-sig suav}). We omit the details.
\end{proof}

\begin{lem}\label{Lem:Bruno}
Denote by $\mathrm{K}(x,z)$ any of the functions $\abs{x}^{2\sigma}F_{-\sigma}(x,z)$, $\abs{z}^{2\sigma}F_{-\sigma}(x,z)$, $0<\sigma\leq1$, or the kernel $x_i\partial_{x_j}F_{-1}(x,z)$. Then
$$\sup_x\int_{\Real^n}\abs{\mathrm{K}(x,z)}~dz\leq C.$$
\end{lem}

\begin{proof}
Consider the function $\abs{x}^{2\sigma}F_{-\sigma}(x,z)$. If $\abs{x}\leq2$ then, by \eqref{F-sig est}, $\displaystyle \abs{x}^{2\sigma}\int_{\Real^n}F_{-\sigma}(x,z)~dz\leq C$. If $\abs{x}>2$, we consider two regions of integration: $\abs{x}<\abs{x-z}$ and $\abs{x}\geq\abs{x-z}$. In the first region, by Lemma \ref{Lem:B-T},
$$\abs{x}^{2\sigma}F_{-\sigma}(x,z)\leq C\int_0^1\left(\frac{1-s}{s}\right)^{n/2}\frac{1}{s^{-\sigma}}~ e^{-C\left[s\abs{x+z}^2+\frac{1}{s}\abs{x-z}^2\right]}~d\mu_{-\sigma}(s)\leq C\Phi(x-z),$$
with $\Phi\in L^1(\Real^n)$. To study the second region of integration, namely $\abs{x}\geq\abs{x-z}$, we use the fact that $\abs{x+z}>\abs{x}$ and we split the integral defining $F_{-\sigma}$ into two intervals: $(0,1/2)$ and $(1/2,1)$. To estimate the part of the integral over the interval $(0,1/2)$ we note that, by using \eqref{est dmu rho} and three different changes of variables, we have
\begin{align*}
  \int_{\abs{x}\geq\abs{x-z}}\int_0^{1/2}G_{t(s)}(x,z)~d\mu_{-\sigma}(s)~dz &\leq C\int_{\abs{x}\geq\abs{x-z}}\int_0^{1/2}\frac{1}{s^{n/2}}~e^{-\frac{1}{4}\left[s\abs{x}^2+\frac{1}{s}\abs{x-z}^2\right]}~\frac{ds}{s^{1-\sigma}}~dz \\
   &= C\abs{x}^{n-2\sigma}\int_{\abs{x}\geq\abs{x-z}}\int_0^{\frac{\abs{x}^2}{2}}\frac{1}{r^{n/2}}~e^{-\frac{1}{4}\left[r+\frac{1}{r}\abs{x}^2\abs{x-z}^2\right]}~\frac{dr}{r^{1-\sigma}}~dz \\
   &= C\abs{x}^{-2\sigma}\int_{\abs{x}^2\geq\abs{w}}\int_0^{\frac{\abs{x}^2}{2}}\frac{1}{r^{n/2}}~e^{-\frac{1}{4}\left[r+\frac{1}{r}\abs{w}^2\right]}~\frac{dr}{r^{1-\sigma}}~dw \\
   &= C\abs{x}^{-2\sigma}\int_0^{\abs{x}^2}\int_0^{\frac{\abs{x}^2}{2}}\frac{1}{r^{n/2}}~e^{-\frac{1}{4}\left[r+\frac{\rho^2}{r}\right]}~\frac{dr}{r^{1-\sigma}}~\rho^{n-1}~d\rho \\
   &\leq C\abs{x}^{-2\sigma}\int_0^\infty\frac{e^{-\frac{r}{4}}}{r^{n/2-\sigma}}\left[\int_0^\infty e^{-\frac{\rho^2}{4r}}\rho^n~\frac{d\rho}{\rho}\right]\frac{dr}{r} \\
   &= C\abs{x}^{-2\sigma}\left[\int_0^\infty e^{-\frac{r}{4}}r^\sigma~\frac{dr}{r}\right]\left[\int_0^\infty e^{-t}t^{n/2}~\frac{dt}{t}\right]=C\abs{x}^{-2\sigma}.
\end{align*}
The integral over the interval $(1/2,1)$ is bounded by
$$\abs{x}^{2\sigma}\int_{1/2}^1(1-s)^{n/2}e^{-C\abs{x}^2}e^{-\frac{\abs{x-z}^2}{C}}~d\mu_{-\sigma}(s)\leq Ce^{-\frac{\abs{x-z}^2}{C}}\in L^1(\Real^n).$$
Hence we get the conclusion for $\mathrm{K}(x,z)=\abs{x}^{2\sigma}F_{-\sigma}(x,z)$. To prove the result for the function $\abs{z}^{2\sigma}F_{-\sigma}(x,z)$, observe that $\abs{z}^{2\sigma}\leq C\left(\abs{z-x}^{2\sigma}+\abs{x}^{2\sigma}\right)$, so we can apply the estimates above. When $\mathrm{F}(x,z)=x_i\partial_{x_j}F_{-1}(x,z)$ we can argue as we did for $\abs{x}^{2\sigma}F_{-\sigma}(x,z)$ above, because of \eqref{x nab F-1}.
\end{proof}

\begin{lem}\label{Lem:Torrea}
For all $1\leq\abs{i}\leq n$, and $0<r_1<r_2\leq\infty$,
$$\sup_x\abs{\int_{r_1<\abs{x-z}\leq r_2}A_iF_{-1/2}(x,z)~dz}\leq C,$$
where $C>0$ is independent of $r_1$ and $r_2$.
\end{lem}

\begin{proof}
By estimate \eqref{Ai F-sig est} given in Lemma \ref{Lem:F-sig est}, it is enough to consider $r_2<1$. From Lemma \ref{Lem:Bruno}, with $\sigma=1/2$, we have that $\displaystyle \int_{\Real^n}x_iF_{-1/2}(x,z)~dz\leq C$. We can write
$$\int_{r_1<\abs{x-z}<r_2}\partial_{x_i}F_{-1/2}(x,z)~dz=\int_{r_1<\abs{x-z}<r_2}I(x,z)~dz+\int_{r_1<\abs{x-z}<r_2}II(x,z)~dz,$$
where
$$I(x,z)=\frac{1}{\Gamma(1/2)}\int_0^1\left(\frac{1-s^2}{4\pi s}\right)^{n/2}e^{-\frac{1}{4}\left[s\abs{x+z}^2+\frac{1}{s}\abs{x-z}^2\right]}\left(-\frac{s}{2}(x_i+z_i)\right)~d\mu_{-1/2}(s).$$
Lemma \ref{Lem:B-T} shows that
$$\abs{I(x,z)}\leq C\int_0^1\left(\frac{1-s}{s}\right)^{n/2}\frac{1}{s^{-1/2}}~e^{-\frac{1}{4}\left[s\abs{x+z}^2+\frac{1}{s}\abs{x-z}^2\right]}~d\mu_{-1/2}(s)\leq\Phi(x-z),
$$
for some integrable function $\Phi$. To deal with $II(x,z)$, we consider the integral
$$\widetilde{II}(x,z)=\frac{1}{\Gamma(1/2)}\int_0^1\left(\frac{1-s^2}{4\pi s}\right)^{n/2}e^{-\frac{1}{4}\left[s\abs{2x}^2+\frac{1}{s}\abs{x-z}^2\right]}\frac{-(x_i-z_i)}{2s}~d\mu_{-1/2}(s),$$
which verifies
$$\abs{\int_{r_1<\abs{x-z}<r_2}\widetilde{II}(x,z)~dz}=0.$$
Therefore, by applying the Mean Value Theorem and some argument parallel to the one used in the proof of Lemma \ref{Lem:B-T}, we have
\begin{align*}
  \abs{II(x,z)-\widetilde{II}(x,z)} &\leq C\int_0^1\left(\frac{1-s}{s}\right)^{n/2}\frac{1}{s^{1/2}}~e^{-\frac{\abs{x-z}^2}{Cs}}\abs{e^{-\frac{1}{4}s\abs{x+z}^2}-e^{-\frac{1}{4}s\abs{2x}^2}}~d\mu_{-1/2}(s) \\
   &\leq C\int_0^1\left(\frac{1-s}{s}\right)^{n/2}e^{-\frac{\abs{x-z}^2}{Cs}}~d\mu_{-1/2}(s)\leq\Psi(x-z),
\end{align*}
for some $\Psi\in L^1(\Real^n)$.
\end{proof}

\textbf{Acknowledgments.} We would like to thank Prof. Jorge J. Betancor, for his comments about Section \ref{Section:Abstracta}, that improved Proposition \ref{Prop:tecnico} and Remark \ref{Rem:K L1}.



\begin{thebibliography}{0}

\bibitem{Bass} R. F. Bass,
{Regularity results for stable-like operators}.
\textit{J. Funct. Anal.}
\textbf{257} (2009), 2693--2722.

\bibitem{Bongioanni-Harboure-Salinas} B. Bongioanni, E. Harboure and O. Salinas,
{Weighted inequalities for negative powers of Schr\"{o}dinger operators},
\textit{J. Math. Anal. Appl.}
\textbf{348} (2008), 12--27.

\bibitem{Bongioanni-Harboure-Salinas Riesz} B. Bongioanni, E. Harboure and O. Salinas,
{Riesz transforms related to Schr\"odinger operators acting on $BMO$ type spaces},
\textit{J. Math. Anal. Appl.}
\textbf{357} (2009), 115--131.

\bibitem{Bongioanni-Torrea} B. Bongioanni and J. L. Torrea,
{Sobolev spaces associated to the harmonic oscillator},
\textit{Proc. Indian Acad. Sci. (Math. Sci.)}
\textbf{116} (2006), 1--24.

\bibitem{Caffarelli-Salsa-Silvestre} L. Caffarelli, S. Salsa and L. Silvestre,
{Regularity estimates for the solution and the free boundary of the obstacle problem for the fractional Laplacian},
\textit{Invent. Math.}
\textbf{171} (2008), 425--461.

\bibitem{Caffarelli-Silvestre Evans-Krylov} L. Caffarelli and L. Silvestre,
{The Evans-Krylov theorem for non local fully non linear equations},
to appear in \textit{Ann. of Math.}
arXiv:0905.1339v1, 2009 (21 pages).

\bibitem{Gilbarg-Trudinger} D. Gilbarg and N. S. Trudinger,
\textit{Elliptic Partial Differential Equations of Second Order},
{Classics in Mathematics},
Springer,
Berlin, 2002.

\bibitem{Harboure-deRosa-Segovia-Torrea} E. Harboure, L. de Rosa, C. Segovia and J. L. Torrea,
{$L^p$--dimension free boundedness for Riesz transforms associated to Hermite functions},
\textit{Math. Ann.}
\textbf{328} (2004), 653--682.

\bibitem{Radha-Thangavelu} R. Radha and S. Thangavelu,
{Multipliers for Hermite and Laguerre Sobolev spaces},
\textit{J. Anal.}
\textbf{12} (2004), 183--191.

\bibitem{Silvestre Thesis} L. Silvestre,
{PhD thesis},
\textit{The University of Texas at Austin},
(2005).

\bibitem{Silvestre CPAM} L. Silvestre,
{Regularity of the obstacle problem for a fractional power of the Laplace operator},
\textit{Comm. Pure Appl. Math.}
\textbf{60} (2007), 67--112.

\bibitem{Stempak-Torrea Acta} K. Stempak and J. L. Torrea,
{Higher Riesz transforms and imaginary powers associated to the harmonic oscillator},
\textit{Acta Math. Hungar.}
\textbf{111} (2006), 43--64.

\bibitem{Stinga-Torrea} P. R. Stinga and J. L. Torrea,
{Extension problem and Harnack's inequality for some fractional operators},
\textit{Comm. Partial Differential Equations}
\textbf{35} (2010), 2092--2122.

\bibitem{Thangavelu} S. Thangavelu,
\textit{Lectures on Hermite and Laguerre Expansions},
Mathematical Notes \textbf{42},
Princeton Univ. Press,
Princeton, NJ, 1993.

\bibitem{Thangavelu2} S. Thangavelu,
{On regularity of twisted spherical means and special Hermite expansions},
\textit{Proc. Indian Acad. Sci.}
\textbf{103} (1993), 303--320.
\end{thebibliography}
\end{document}